\documentclass[12pt]{amsart}
\usepackage{amssymb}
\usepackage{enumerate}

        \usepackage{graphicx}
        \usepackage[usenames,dvipsnames,svgnames,table]{xcolor}
        \usepackage[a4paper]{geometry}

\newcommand{\N}{{\mathbb N}}
\newcommand{\C}{{\mathbb C}}

\newcommand{\wt}{\widetilde}
\newcommand{\wh}{\widehat}
\newcommand{\wand}{wandering domain}

\newcommand{\tef}{transcendental entire function}

\newcommand{\nhd}{neighbourhood}
\newcommand{\sconn}{simply connected}
\newcommand{\mconn}{multiply connected}

\newcommand{\sw}{spider's web}

\newcommand{\spl}{strongly polynomial-like}
\newcommand\qfor{\;\;\text{for }}

\newcommand\qas{\quad\text{as }}
\newcommand\qand{\quad\text{and}\quad}

\newcommand{\kfc}{I^{\!+\!}(f)}

\newcommand{\vfp}{V^{\!+\!}(f)}
\newcommand{\vfpr}{V_R^{\!+\!}(f)}

\newcommand{\ifm}{I^{\!+\!}(f, (m^n(r)))}
\newcommand{\kfm}{K(f, (m^n(r)))}
\newcommand{\ifmd}{I^{\!+\!}(f, (m^n(r')))}
\newcommand{\kfmd}{K(f, (m^n(r')))}
\newcommand{\ifmh}{\bigcup_{\ell \in \N_0} f^{- \ell}(\ifm)}
\newcommand{\mt}{\widetilde{m}}

\theoremstyle{plain}
\newtheorem{theorem}{Theorem}[section]

\newtheorem*{theorem*}{Theorem}
\newtheorem*{proposition*}{Proposition}

\newtheorem{lemma}[theorem]{Lemma}
\theoremstyle{definition}

\theoremstyle{remark}
\newtheorem*{remark*}{Remark}
\newtheorem*{remarks*}{Remarks}

\theoremstyle{problem}

\theoremstyle{example}
\newtheorem{example}[theorem]{Example}
\newtheorem*{example*}{Example}
\theoremstyle{question}
\newtheorem{question}[theorem]{Question}
\theoremstyle{questions}
\newtheorem*{questions*}{Questions}

{\begin{list}{}%
         {\setlength{\leftmargin}{#1}}%
         \item[]%
}
{\end{list}}

\begin{document}


\title[Iterating the minimum modulus]{The iterated minimum modulus \\and conjectures of Baker and Eremenko}

\author{J.W. Osborne, P.J. Rippon \and G.M. Stallard}
\address{Department of Mathematics and Statistics \\
   The Open University \\
   Walton Hall\\
   Milton Keynes MK7 6AA\\
   UK}
\email{john.osborne@open.ac.uk, phil.rippon@open.ac.uk, gwyneth.stallard@open.ac.uk}



\thanks{The last two authors were supported by the EPSRC grant EP/K031163/1.}


\begin{abstract}
In transcendental dynamics significant progress has been made by studying points whose iterates escape to infinity at least as fast as iterates of the maximum modulus. Here we take the novel approach of studying points whose iterates escape at least as fast as iterates of the {\it minimum} modulus, and obtain new results related to Eremenko's conjecture and Baker's conjecture, and the rate of escape in Baker domains. To do this we prove a result of wider interest concerning the existence of points that escape to infinity under the iteration of a positive continuous function.
\end{abstract}
\maketitle
\begin{center}
{\it For Walter Hayman on the occasion of his ninetieth birthday.}\\
\end{center}

\section{Introduction}
\label{intro}
\setcounter{equation}{0}

Denote the $ n $th iterate of a function $ f $ by $ f^n $, for $ n \in \N$. If $f$ is a transcendental entire function then the \textit{Fatou set} $ F(f) $ is the set of points $ z \in \C $ such that the family of functions $ \lbrace f^n : n \in \N \rbrace $ is normal in some \nhd\ of $ z $ and the \textit{Julia set} $ J(f) $ is the complement of $ F(f) $. We refer to \cite{aB, wB93, CG, Mil} for the fundamental properties of these sets and an introduction to complex dynamics.

The \emph{escaping set} $ I(f) = \lbrace z \in \C : f^n(z) \to \infty \text{ as } n \to \infty \rbrace $ was first studied for a general \tef\ $ f $ by Eremenko \cite{E}. In recent years, the \emph{fast escaping set} $ A(f) $ has played a significant role in transcendental dynamics, for example, in progress on \emph{Eremenko's conjecture}, that all the components of $ I(f) $ are unbounded, and on \emph{Baker's conjecture}, that if $ f $ has order at most~$ 1/2, $ minimal type, then all the components of $ F(f) $ are bounded. Despite many partial results, both conjectures remain open.


The set $A(f)$ was introduced in~\cite{BH99} and consists of those points whose iterates under $ f $ eventually grow at least as fast as iterates of the maximum modulus, $ M(r) = \max_{|z|=r} |f(z)| $. It can be defined as follows:
\[
A(f) = \bigcup_{\ell \in \N_0} f^{-\ell} (A_R(f)), \, \, \textrm{where} \, A_R(f) = \{ z: |f^n(z)| \geq M^n(R), \textrm{ for } n \in \N\}.
\]
Here  $ \N_0 = \N \cup \{0\}, $  $ M^n(r) $ denotes the $ n $th iterate of the function $ r \mapsto M(r) $, and $ R > 0 $ is such that $ M^n(R) \to \infty $ as $ n \to \infty. $ Note that there always exists $ R>0 $ such that, for $ r \geq R, $ we have $ M(r)>r $ and hence $ M^n(R) \to \infty $ as $ n \to \infty $, and the definition of $ A(f) $ is independent of the choice of such an $ R $.

The set $A(f)$ has many strong properties \cite{RS05, RS10a} and was used in \cite{RS14} to show that the escaping set $I(f)$ is either connected or has infinitely many unbounded components. See \cite{R07} and \cite{R3S} for other partial results on Eremenko's conjecture.

In this paper, we study those points whose iterates under $ f $ eventually grow at least as fast as iterates of the \emph{minimum} modulus, $  m(r) = \min_{|z|=r} |f(z)| $. Replacing $ M(r) $ by $ m(r) $ in the definition of $ A(f) $ does not in general yield a subset of $I(f)$.  Indeed, if the function $ m(r) $ is bounded, then its iterates tell us nothing about~$ I(f) $; this is the case, for example, when $ f $ is in the Eremenko-Lyubich class $\mathcal{B}$ of \tef s with a bounded set of singular values (that is, critical values and asymptotic values).

It turns out, however, that iterating the minimum modulus is of significant interest for the many entire functions with the property that
\begin{equation}\label{minmodprop}
\text{there exists } r > 0 \text{ with } m^n(r) \to \infty \text{ as } n \to \infty.
\end{equation}
We introduced condition (\ref{minmodprop}) in \cite{ORS} in the context of investigating the connectedness properties of the set $ \kfc $ of points at which the iterates of $ f $ form an unbounded sequence.  In this paper, we make a deeper study of the condition~\eqref{minmodprop}, and in this way obtain new results related to Eremenko's conjecture and Baker's conjecture, and about the rate of escape in Baker domains (defined in Section~\ref{unbFat}).

In order to work with~\eqref{minmodprop} and, in particular, to identify \tef s for which (\ref{minmodprop}) holds, it is useful to introduce the maximal function
\[ \mt(r) := \max_{0 \leq s \leq r} m(s), \qfor r \in [0, \infty).\]
We show that~(\ref{minmodprop}) is true if and only if
\begin{equation}\label{mtequiv}
\text{there exists } R > 0 \text{ such that } \mt(r) >r, \text{ for } r \geq R.
\end{equation}
In fact, in Section~\ref{Real} we prove a general result about escaping points of positive continuous functions, which is of wider interest.

Using the condition~\eqref{mtequiv} and several classical results about the size of the minimum modulus, we are able to show that condition (\ref{minmodprop}) holds for many classes of entire functions.  The terminology used  in the following result  is explained in Section~\ref{Functions}.
\begin{theorem}
\label{funcs}
Let $ f $ be a \tef. Then (\ref{minmodprop}) holds if
\begin{enumerate}[(a)]
\item $ f $ is of order less than $ 1/2 $, or
\item $ f $ has finite order and Fabry gaps, or
\item $ f $ has Hayman gaps, or
\item $ f $ exhibits the pits effect, as defined by Littlewood and Offord, or
\item $ f $ has a multiply connected Fatou component.
\end{enumerate}
\end{theorem}

In Section~\ref{examples} we give several further examples of familiar entire functions that satisfy the condition \eqref{minmodprop}, such as $f(z)=2z(1+e^{-z})$.

It is reasonable to expect that, for functions satisfying (\ref{minmodprop}), the behaviour of $ m^n(r) $ will depend on the choice of $ r $.  Indeed, we shall see in Section~\ref{Real} that if \eqref{minmodprop} holds, then $ m^n(r) $ can tend to infinity arbitrarily slowly.


In contrast, it is clear that the \emph{fastest} rate at which $ m^n(r) $ can tend to infinity must be related to the growth of $ \mt^n(r) $; indeed, this rate is always attained.

\begin{theorem}
\label{fast}
Let $ f $ be a \tef\ such that (\ref{minmodprop}) holds. Then there  exists  $R>0$ such that
\begin{equation}\label{mtcond1}
\mt^n(R) \to \infty \textrm{ as } n \to \infty,
\end{equation}
 and, for any such $R$, there exists  $ r \geq R $ such that
\begin{equation}
\label{mtcond}
 m^n(r) \geq \mt^n(R), \textrm{ for } n \in \N.
\end{equation}
Moreover, for every $ r $ and $ R $ satisfying  (\ref{mtcond1}) and (\ref{mtcond}),  the  set
\begin{align*}
\bigcup_{\ell \in \N_0} f^{-\ell} \left( \{ z: |f^n(z)| \geq \mt^n(R), \textnormal{ for } n \in \N\} \right)
\end{align*}
is independent of the choice of $R$ and is equal to the set
\begin{align*}
\bigcup_{\ell \in \N_0} f^{-\ell} \left( \{ z: |f^n(z)| \geq m^n(r), \textnormal{ for } n \in \N\} \right).
\end{align*}
\end{theorem}
It follows that, if (\ref{minmodprop}) holds, then the set
\[ V(f) = \bigcup_{\ell \in \N_0} f^{-\ell} (V_R(f)), \,\, \,  \textrm{where} \, V_R(f) = \{ z: |f^n(z)| \geq \mt^n(R), \textrm{ for } n \in \N\},\]
is  well defined  and independent of $ R $, provided $R$ is so large that $ \mt(r) > r $ for $ r \geq R $. Moreover, $V(f)$ is completely invariant under~$f$.

Since $ M^n(r) \geq \mt^n(r) $ for $ n \in \N$, we always have $ A(f) \subset V(f) $ and thus $ V(f) \neq \emptyset. $ Perhaps surprisingly, there are many classes of functions for which $V(f) = A(f)$. In fact, as we show in Section~\ref{setV}, all the functions listed in Theorem~\ref{funcs} have this property provided they satisfy a certain regularity condition.

The following result shows that the conclusions of both Eremenko's conjecture and Baker's conjecture hold for any function for which $V(f) = A(f)$.

\begin{theorem}
\label{erembak}
Let $ f $ be a \tef\ such that (\ref{minmodprop})  holds, so that the set $ V(f) $ is well defined.
\begin{enumerate}[(a)]
\item We have $ V(f) = A(f) $ if and only if there exist $ r \geq R > 0 $ such that
\begin{align*}
m^n(r) \geq M^n(R) \textrm{ for } n \in \N, \qand M^n(R) \to \infty \textrm{ as } n \to \infty.
\end{align*}
\item If the equivalent conditions in part~(a) hold, then
\begin{enumerate}[(i)]
\item $ V(f) $ and $ I(f) $ are spiders' webs, and
\item $ F(f) $ has no unbounded components.
\end{enumerate}
\end{enumerate}
\end{theorem}

Here, a set $ E $ is a \emph{\sw} if it is connected and there exists a sequence $ (G_n) $ of bounded, \sconn\ domains such that
\[ G_n \subset G_{n+1}, \,\, \partial G_n \subset E, \,\, \textrm {for }  n \in \N, \qand \bigcup_{n \in \N} G_n = \C. \]
We prove part~(b) of Theorem \ref{erembak} by showing that if the conditions in part~(a) hold, then the set $A_R(f)$, defined earlier, is a \sw,                                                                                                                                                                                                                                                                                                                                                                                                                                                                                                                                                                                                                                        a property that has several strong consequences \cite{RS10a}. This is, in essence, the approach used to prove all partial results on Baker's conjecture prior to the recent papers \cite{NRS,RS13}; see \cite{iB81,AH, HM, RS09a} and the discussion in \cite[Introduction]{RS13}.

In view of Theorem~\ref{erembak}, it is natural to ask the following.

\begin{question}\label{Qn}
Let $ f $ be a \tef\ such that (\ref{minmodprop}) holds.  Do the following conclusions hold under a weaker hypothesis than $V(f) = A(f)$?
\begin{enumerate}
\item $ V(f) $ and $ I(f) $ are spiders' webs;
\item $ F(f) $ has no unbounded components.
\end{enumerate}
\end{question}

Various recent results suggest that a significant weakening of the hypothesis here that $V(f) = A(f)$ is indeed plausible. For example, by \cite{RS13b} there exist entire functions of order less than $1/2$ (so \eqref{minmodprop} holds) for which $ A_R(f) $ is not a \sw. Thus, by the discussion after Theorem~\ref{erembak}, we have $  V(f) \neq A(f), $ and yet it can be shown using \cite[Theorems~1.1 and 2.3]{RS13} that properties~(1) and~(2) hold.  Also, the most recent work on Baker's conjecture \cite[Corollary~1.2]{NRS} shows that property~(2) holds whenever $ f $ is a real function of order at most $ 1/2 $, minimal type, with only real zeros.  Note, however, that there are \tef s of order greater than $ 1/2 $ for which (\ref{minmodprop}) holds and for which the Fatou set has an unbounded component (see Example~\ref{Exfirst}), so some hypothesis is needed in addition to \eqref{minmodprop}.

As a step towards making progress on Question~\ref{Qn}, we prove two main results, the first of which is a refinement of \cite[Theorem 1.2]{ORS}. There we showed that if \eqref{minmodprop} holds, then $ \kfc $ is connected, where $ \kfc $ is the set of points at which the iterates of~$ f $ form an unbounded sequence.   We now  define $ \vfp $ to be the set derived from $ V(f) $ in the same way that $ \kfc $ is derived from $ I(f) $ \---\ that is, by adding those points for which only a \emph{subsequence} of iterates satisfies its defining property. Thus
\[ \vfp = \bigcup_{\ell \in \N_0} f^{-\ell} (\vfpr), \]
where
\[ \vfpr = \{ z: \exists \,\, (n_j) \textrm{ such that } |f^{n_j}(z)| \geq \mt^{n_j}(R), \textrm{ for } j \in \N\}.\]
Here, $ (n_j) $ is a strictly increasing sequence of positive integers that in general depends on~$ z $, and $ R>0 $ is such that $ \mt(r) > r $ for $ r \geq R. $

The following is our refinement of \cite[Theorem 1.2]{ORS}.


\begin{theorem}
\label{weaksw}
Let $ f $ be a \tef\ such that (\ref{minmodprop}) holds. Then the set $\vfp$, and also the set $\kfc$, is a weak spider's web.
\end{theorem}

By a \emph{weak \sw} we mean a connected set whose complement contains no  unbounded closed connected sets. The name arises from the fact that a \sw\ has the stronger property that its complement contains no  unbounded connected sets; see Figure~\ref{weakSWpic} for an illustration of a set that is a weak spider's web but not a spider's web.

\begin{figure}[htb]
\begin{center}
\includegraphics[width=10cm]{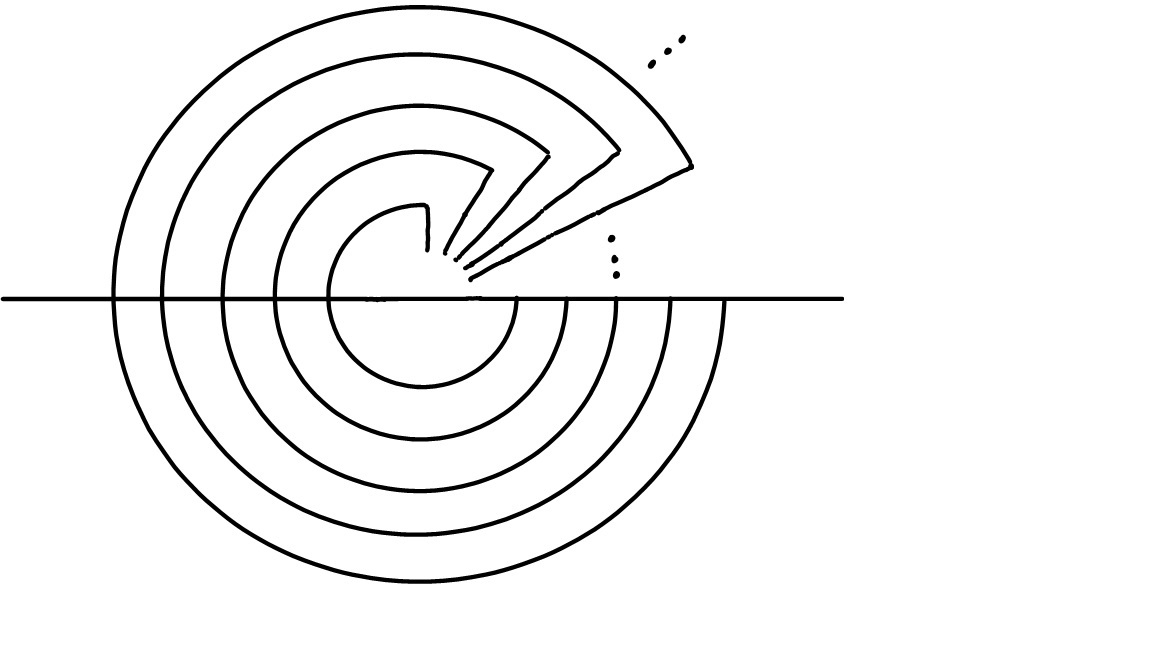}
\caption{A weak spider's web that is not a spider's web}
\label{weakSWpic}
\end{center}
\end{figure}

{\it Remark} \;If we could replace $\vfp$ by $V(f)$, or $\kfc$ by $I(f)$, in Theorem~\ref{weaksw}, then this would show that the conclusion of Eremenko's conjecture holds for all entire functions satisfying \eqref{minmodprop}. Indeed, if $V(f)$ is connected, then $I(f)$ is connected; see Theorem~\ref{VandI}.

In order to prove Theorem~\ref{weaksw}, we show that a major consequence of the condition \eqref{minmodprop} is that certain sufficiently long continua must contain a point of $\vfp$; see Theorem~\ref{notlong} for a detailed statement. We also use this fact about long continua meeting $V^+(f)$ to obtain our second result related to Question~\ref{Qn}. This second result strengthens several results of Zheng \cite{Z0}, which relate the dynamical behaviour of a \tef\ $f$ in an unbounded  component of $ F(f) $  to the size of the minimum modulus of~$f$, and hence are related to Baker's conjecture.
\newpage
First we restate the results of Zheng below, using the notation of this paper. (See Section~\ref{unbFat} for the definition of a wandering domain.)

\begin{itemize}
\item[(a)] If
\[m(r)>r \text{ for an unbounded sequence of values of } r\]
and~$U$ is an unbounded component of~$F(f)$, then $U\subset \kfc$; see \cite[Theorem~2]{Z0}.
\item[(b)] If \[\limsup_{r\to\infty}\dfrac{m(r)}{r}=\infty\] and $U$ is an unbounded component of $F(f)$, then $U$ is a wandering domain and $U\subset \kfc$; see \cite[Theorem~1]{Z0}.
\item[(c)] If $f$ has order less than 1/2 and~$U$ is an unbounded component of $F(f)$, then~$U$ is a wandering domain and $U\subset Z^+(f)$; see \cite[Theorem~3]{Z0}.
\end{itemize}

The set $Z^+(f)$ in result~(c) is defined as follows:
\[ Z^+(f):= \left\{ z  :  \exists \,\, (n_j) \textrm{ such that } \dfrac{\log^+ \log^+|f^{n_j}(z)| }{n_j} \to \infty \textrm{ as } j \to \infty \right\},\]
where $ (n_j) $ is a strictly increasing sequence of positive integers. The set $ Z^+(f) $ is formed from the set $ Z(f) $ of points whose iterates `zip' to $ \infty $ (for which, see \cite{RS00}) by adding those points at which only a subsequence of iterates has this property.

It is clear that Zheng's result~(c) is related to Baker's conjecture. To see that~(a) and~(b) are also related to this conjecture, note that the hypotheses about $m(r)$ in~(a) and~(b) both hold whenever~$f$ has order at most 1/2, minimal type, by a theorem of Heins \cite{mH48}. Also, note that the condition on $m(r)$ in~(a) ensures that~$f$ is a strongly polynomial-like function; see \cite{O12a}, where a generalisation of Zheng's result~(a) to all strongly polynomial-like functions is given.

We use the property of long continua meeting $\vfp$ to prove the following theorem, which gives more detailed information than Zheng's results~(a) and~(b) whenever \eqref{minmodprop} holds, and strengthens Zheng's result~(c).
\newpage
\begin{theorem}
\label{fatou}
Let $ f $ be a \tef\ and let $U$ be an unbounded component of $F(f)$.
\begin{itemize}
\item[(a)] If \eqref{minmodprop} holds, then $U \cap \vfp \neq \emptyset$, and hence $U$ is either a Baker domain (or preimage of a Baker domain) or a wandering domain. \\
If, in addition,
\[
\liminf_{r \to \infty} \frac{\mt(r)}{r} > 1,
\]
then $U \subset \vfp$.
\item[(b)] If
\[
\lim_{r \to \infty} \frac{\mt(r)}{r}  = \infty,
\]
then $U$ is a wandering domain and $U\subset \vfp$.
\item[(c)] If
\[
\lim_{r \to \infty} \frac{\log \mt(r)}{\log r} =\infty,
\]
then $U$ is a wandering domain and $U\subset \vfp\subset Z^+(f)$.
\end{itemize}
\end{theorem}

The hypothesis on $\mt$ in part~(c) of this result holds for many classes of entire functions, including those of order less than 1/2;  see Section \ref{Functions} and, in particular, Lemma~\ref{suffcon}.

Finally, we also provide new information about the rate of escape that occurs in certain Baker domains.

\begin{theorem}
\label{BD}
Let $f$ be a \tef\ and let $U$ be an invariant Baker domain of~$f$.
\begin{itemize}
\item[(a)] If \eqref{minmodprop} holds, then there exists $R>0$ and, for each $z \in U$, a constant $C(z) > 1$ such that
\[
|f^n(z)| \geq \mt^n(R) / C(z),\;\; \text{for } n \in \N,
\]
and
\[
\mt^n(R)\to\infty \;\;\text{as } n\to \infty.
\]
\item[(b)] If, in addition,
\[
\liminf_{r \to \infty} \frac{\mt(r)}{r} > 1,
\]
then $U \subset V(f)$.
\end{itemize}
\end{theorem}

The organisation of this paper is as follows.  In Section \ref{Real}, we prove our results about the existence of escaping points of positive continuous functions; in particular, Theorem~\ref{fullequiv} gives several equivalent conditions for the existence of such points.  In Section~\ref{Functions}, we use Theorem \ref{fullequiv} to deduce  Theorem  \ref{funcs} on classes of functions for which condition (\ref{minmodprop}) holds. In Sections~\ref{setV} and~\ref{VequalsA}, we prove Theorems \ref{fast} and \ref{erembak} concerning the set $ V(f) $ and show that $V(f)=A(f)$ for the functions listed in Theorem~\ref{funcs} provided they satisfy a mild regularity condition.  The proofs of Theorem \ref{weaksw} and related results are given in Section \ref{WeakSWRes}, and the proofs of Theorems~\ref{fatou} and~\ref{BD} are given in Section~\ref{unbFat}.  Finally, in Section~\ref{examples}, we give some examples to illustrate our results.

\section{Escaping points of positive continuous functions}
\label{Real}
\setcounter{equation}{0}

%

Let $ \varphi: [0, \infty) \to [0, \infty) $ be a continuous function and define the maximal function~$\widetilde{\varphi}$ by
\[ \widetilde{\varphi}(t) := \max_{0 \leq s \leq t} \varphi(s), \qfor t \in [0, \infty). \]
The following result gives a number of conditions equivalent to the existence of escaping points of $\varphi$. In particular, this theorem justifies our assertion in the introduction that
\[
\text{there exists } r > 0 \text{ with } m^n(r) \to \infty \text{ as } n \to \infty
\]
if and only if
\[
\text{there exists } R > 0 \text{ such that } \mt(r) >r, \text{ for } r \geq R.
\]

\begin{theorem}
\label{fullequiv}
Let $ \varphi: [0, \infty) \to [0, \infty) $ be continuous. Then the following statements are equivalent.
\begin{enumerate}[(a)]
\item There exists $ t > 0 $ such that $ \varphi^n(t) \to \infty $ as $ n \to \infty. $
\item There exists $ t' > 0 $ such that the set $ \lbrace \varphi^n(t'): n \in \N_0 \rbrace $ is unbounded.
\item There exists $ T>0 $ such that $ \widetilde{\varphi}(t) > t $, for $ t \geq T $.
\item There exist $ t \geq T > 0 $ such that
\[ \varphi^n(t) \textrm{ and } \widetilde{\varphi}^{\, n}(T) \textrm{ increase strictly with } n \textrm{ to } \infty, \]
and
\[ \varphi^n(t) \in [ \, \widetilde{\varphi}^{\, n}(T) , \widetilde{\varphi}^{\, n+1}(T) \, ], \qfor n \in \N_0. \]
\item There exists a sequence $ (t_n) $ of positive real numbers such that $ t_n \to \infty $ as $ n \to \infty $ and
\begin{equation*}
\varphi(t_n) \geq t_{n+1}, \quad \textrm{ for }  n \in \N_0.
\end{equation*}
\end{enumerate}
\end{theorem}

{\it Remarks }\;1. It follows from a recent result of Short and Sixsmith \cite[Theorem~1.6]{SS} that if $ t>0 $ satisfies the condition in Theorem \ref{fullequiv} part~(a), then every open interval containing $ t $ includes uncountably many escaping points.

2. The condition in part~(d) of Theorem~\ref{fullequiv} shows that there always exist points that escape at the fastest possible rate. We use this fact in Section~\ref{setV} in the proof of Theorem~\ref{fast}.

3. In Section~\ref{WeakSWRes}, we use Theorem~\ref{fullequiv} to show that if a \tef\ $f$ satisfies \eqref{minmodprop}, then any curve that tends to~$\infty$ and is invariant under~$f$ must meet $I(f)$.

In the introduction we also stated that if \eqref{minmodprop} holds, then $m^n(r)$ can tend to $\infty$ arbitrarily slowly. This follows from our next result because any \tef\ either has infinitely many zeros or has asymptotic value~0, by Iversen's theorem, so
\[
\liminf_{r\to\infty} m(r)=0.
\]

\begin{theorem}
\label{realslow}
Let $ \varphi: [0, \infty) \to [0, \infty) $ be continuous and suppose there exists $ t > 0 $ such that $ \varphi^n(t) \to \infty $ as $ n \to \infty. $  Suppose also that
\begin{align}
\label{limf}
\liminf_{t \to \infty} \varphi(t) < \infty.
\end{align}
If $ a = (a_n) $ is a positive sequence such that $ a_n \to \infty $ as $ n \to \infty $, then there exist $ t_a > 0 $ and $ N_a \in \N $ such that $ \varphi^n(t_a) \to \infty $ as $ n \to \infty $ and
\begin{align}
\label{slowseq}
\varphi^n(t_a) \leq a_n,  \quad \textrm{  for  } n \geq N_a.
\end{align}
\end{theorem}

The key techniques used in the proofs of Theorems~\ref{fullequiv} and \ref{realslow} are given in the following lemma.
\begin{lemma}
\label{key}
Let $ \varphi $ and $ \widetilde{\varphi} $ be as defined before the statement of Theorem \ref{fullequiv}, and suppose there exists $ T>0 $ such that $ \widetilde{\varphi}(t) > t, $ for $ t \geq T $.  Then
\begin{enumerate}[(a)]
\item the sequence $ \widetilde{t}_n : = \widetilde{\varphi}^{\, n}(T), \, n \in \N_0, $ is strictly increasing and tends to $ \infty $;
\item if $ E_n=[\, \widetilde{t}_n, \widetilde{t}_{n+1} \, ], $ for $ n \in \N_0, $ then
\[ \varphi(E_n) \supset E_{n+1}, \qfor n \in \N_0; \]
\item if, in addition,
\[ \liminf_{t \to \infty} \varphi(t) < \infty, \]
then there is a subsequence $ (E_{n(k)}), \, k \in \N_0, $ such that
\[ \varphi(E_{n(k)}) \supset \bigcup_{n = n(0)}^{n(k)} E_n, \qfor k \in \N_0. \]
\end{enumerate}
\end{lemma}

\begin{proof} The hypothesis of the lemma clearly implies part~(a).
It follows from the definitions of $ \wt{\varphi} $ and $ \wt{t}_n $ that
\begin{align}
\label{T1}
\varphi(\wt{t}_n) \leq \wt{\varphi}(\wt{t}_n) = \wt{t}_{n + 1}, \qfor n \in \N_0,
\end{align}
and
\begin{align}
\label{T2}
\wt{t}_{n + 2} = \wt{\varphi}(\wt{t}_{n+1}) = \varphi(t), \qfor \textrm{some } t \leq \wt{t}_{n + 1}.
\end{align}
Now we cannot have $ t < \wt{t}_n, $ for otherwise, using part (a),
\[ \varphi(t) \leq \wt{\varphi}(t) \leq \wt{\varphi}(\wt{t}_n) = \wt{t}_{n + 1} < \wt{t}_{n + 2},\]
which contradicts (\ref{T2}).  Thus $ t \in [\, \widetilde{t}_n, \widetilde{t}_{n+1} \, ] = E_n $ and, together with (\ref{T1}), this shows that
\[ \varphi(E_n) \supset E_{n+1}, \qfor n \in \N_0, \]
which proves part (b).

Part (c) follows  by choosing  $ n(0) $ so that $ \wt{t}_{n(0)} > \liminf_{t \to \infty} \varphi(t). $
\end{proof}

We now give the proofs of Theorems~\ref{fullequiv} and \ref{realslow}. The proof of Theorem \ref{fullequiv} makes use of a simple topological result, which is widely used, going back at least to \cite{E}, and stated explicitly in \cite[Lemma 1]{RS09}.   We quote here a version of this result which we also need later in the paper; compare \cite[Lemma 3.1]{Sixsmithmax}.
\begin{lemma}
\label{RSlemm}
Let $(E_j)_{j \in \N_0}$ be a sequence of compact sets in $ \C $, $(m_j)_{j \in \N_0}$ be a sequence of positive integers and $f: \C \to \wh{\C} $ be a continuous function such that
\[ f^{m_j}(E_j ) \supset E_{j+1} , \qfor  j \in \N_0. \]
Then there exists $\zeta\in E_0$ such that
\begin{equation*}
\label{feq}
f^{m_0+m_1+\cdots +m_n}(\zeta) \in E_{n+1}, \qfor n \in \N_0.
\end{equation*}
\end{lemma}

\begin{proof}[Proof of Theorem \ref{fullequiv}]
We show that (a) $ \Rightarrow $ (b) $ \Rightarrow $ (c) $ \Rightarrow $ (d) $ \Rightarrow $ (e), and then that (e) $ \Rightarrow $ (c).  Since it is clear that (d) $ \Rightarrow $ (a), this will prove the theorem.

It is obvious that (a) $ \Rightarrow $ (b) and that (d) $ \Rightarrow $ (e).

Suppose~(b) holds, and that $ t' > 0 $ is such that the set $ \lbrace \varphi^n(t'): n \in \N_0 \rbrace $ is unbounded.  Then, for $ t \geq t' $, there is an integer $ N = N(t) \in \N $ such that $ \varphi^{N-1}(t') \leq t $ and $ \varphi^N(t') > t. $  It follows that
\[
\wt{\varphi}(t) \geq \wt{\varphi}(\varphi^{N-1}(t')) \geq \varphi^N(t') > t,
\]
which proves~(c) with $ T = t' $.

Now suppose (c) holds, and that $ T>0 $ is such that $ \widetilde{\varphi}(t) > t $ for $ t \geq T $.  Define
\[ E_n=[\, \widetilde{\varphi}^{\, n}(T), \widetilde{\varphi}^{\, n+1}(T) \, ] , \qfor n \in \N_0. \]
Then it follows by Lemma \ref{RSlemm} and Lemma~\ref{key} part~(b) that there is a point $ t \in E_0 $ such that $ \varphi^n(t) \in E_n, $ for $ n \in \N_0, $ and therefore $ \varphi^n(t) \geq \wt{\varphi}^{\, n}(T) \to \infty, \textrm{  as  } n \to \infty. $  Moreover, the sequence $ ( \varphi^n(t) ) $, $ n \in \N_0, $ is strictly increasing, since otherwise it would eventually be constant.  This proves~(d).

It remains to prove that (e) $ \Rightarrow $ (c).  Suppose (e) holds, and let $ (t_n) $ be a positive sequence that tends to $ \infty $ as $ n \to \infty $, with $ \varphi(t_n) \geq t_{n+1} $ for $  n \in \N_0.$  Set $ T = t_0. $  If $ t \geq T $, then $ t \in [ \, t_n, t_{n+1} \, ) $ for some $ n \in \N_0 $, so
\[
\wt{\varphi}(t) \geq \varphi(t_n) \geq t_{n+1} > t,
\]
which proves (c) and completes the proof of the theorem.
\end{proof}

\begin{proof}[Proof of Theorem \ref{realslow}]
 By Theorem~\ref{fullequiv},  there exists $T>0$ such that $\widetilde{\varphi}(t) > t$ for $t \geq T$. Thus it follows from Lemma~\ref{key} that the sequence  $\wt{t}_n = \widetilde{\varphi}^n(T)$  is strictly increasing and that, if $E_n = [\, \widetilde{t}_n, \widetilde{t}_{n+1} \, ] $ for $ n \in \N_0$, then
\begin{align}
\label{P1}
\varphi(E_n) \supset E_{n+1}, \qfor n \in \N_0.
\end{align}
Furthermore, since (\ref{limf}) holds, it follows from Lemma~\ref{key} part~(c) that there is a subsequence $ (E_{n(k)}), \, k \in \N_0, $ such that
\begin{align}
\label{P2}
\varphi(E_{n(k)}) \supset E_{n(k)}, \qfor k \in \N_0.
\end{align}
 We now construct a new sequence of intervals by selecting terms from the sequence $ (E_n) $ in such a way that we can apply Lemma \ref{RSlemm} and hence deduce the existence of $ t_a > 0 $ and $ N_a \in \N $ with the properties in the statement of the theorem.  The basic idea is that, in this new sequence, we repeat each of the intervals $ E_{n(k)} $ sufficiently often that, by using property (\ref{P2}), we `slow down' the rate at which $ \varphi^n(t_a) \to \infty $ to ensure that (\ref{slowseq}) is satisfied.

Without loss of generality, we can assume that $ a = (a_n) $ is increasing.  Let the  new  sequence $ (F_m), \, m \in \N_0, $ consist of all the intervals $ E_n $ for $ n \geq n(0) $, taken in order of increasing $ n $, but with each interval $ E_{n(k)}, \, k \in \N_0, $ repeated $ m(k) $ times, where $ m(k) $ is so large that
\[ \bigcup_{n = n(0)}^{n(k+1)} E_n \subset [ \, 0, a_{m(k)} \, ] .\]
Since the interval $ E_{n(k)} $ is repeated $ m(k) $ times in the  new  sequence $ (F_m) $, it follows that if $ F_m = E_n, $ where $n(k)+1 \leq n \leq n(k+1)$, then $ m \geq m(k). $  Thus, for such~$ n $, we have
\[ F_m = E_n \subset [ \, 0, a_{m(k)} \, ] \subset [ \, 0, a_m \, ], \]
and it follows that
\begin{align}
\label{slowed}
F_m \subset [ \, 0, a_m \, ], \qfor m \geq m(0).
\end{align}
Now by (\ref{P1}) and (\ref{P2}) we have
\[ \varphi(F_m) \supset F_{m+1}, \qfor m \in \N_0, \]
so applying Lemma \ref{RSlemm} we obtain $ t_a \in F_0 $ such that $ \varphi^m(t_a) \in F_m,$ for $ m \in \N_0$. Clearly  $ \varphi^m(t_a) \to \infty $ as $ m \to \infty $, and it follows from (\ref{slowed}) that
\begin{align*}
\varphi^m(t_a) \leq a_m,  \quad \textrm{  for  } m \geq m(0).
\end{align*}
This completes the proof.
\end{proof}

\section{Classes of functions for which condition (\ref{minmodprop}) holds}
\label{Functions}
\setcounter{equation}{0}

In this section  we prove Theorem \ref{funcs}, which lists a number of large classes of \tef s for which condition (\ref{minmodprop}) holds.
A key step in the proof of parts~(a)--(d) of Theorem~\ref{funcs} is the following simple lemma, which will also be needed in Section~\ref{unbFat}.

\begin{lemma}
\label{suffcon}
Suppose $ f $ is a \tef\ with the property that there exist $ C>1 $ and $ R_0>0 $ such that, for $ r \geq R_0 $,
\begin{equation}\label{maxmin}
\text{there exists } s  \in (r, r^C) \text{ with } m(s) \geq M(r).
\end{equation}
Then condition (\ref{minmodprop}) is satisfied and, more strongly,
\begin{equation}\label{mtgrowth}
\frac{\log \mt(r)}{\log r} \to \infty\quad \text{as } r\to \infty;
\end{equation}
in particular, for~$r$ sufficiently large, we have
\begin{equation}\label{mtzips}
\frac{\log^+\log^+ \mt^n(r)}{n}\to \infty\quad \text{as } n\to \infty.
\end{equation}
\end{lemma}

\begin{proof}
By \eqref{maxmin}, we have
\begin{equation}\label{mtbig}
 \mt(r^C) \geq m(s) \geq M(r),\qfor r \geq R_0.
\end{equation}
Then \eqref{mtgrowth} follows from \eqref{mtbig} together with the well known fact that if~$f$ is a \tef, then
\begin{equation}\label{growthM(r)}
\frac{\log M(r)}{\log r} \to \infty\quad \text{as } r\to \infty,
\end{equation}
and (\ref{minmodprop}) follows from \eqref{mtgrowth} and Theorem \ref{fullequiv}. Finally, \eqref{mtzips} follows easily from \eqref{mtgrowth}.
\end{proof}

To prove Theorem \ref{funcs}, we first show that each of the classes (a)--(d) in the theorem meets the condition in Lemma \ref{suffcon}.  The same arguments arose in the context of functions satisfying \cite[Corollary 8.3(a)]{RS10a} but we include a summary here for completeness.  Class~(e) in the theorem has a separate proof.

We first define some terms needed here and later in the paper.   The \textit{order} $ \rho $, \emph{lower order} $ \lambda $ and \textit{type} $ \tau $ of an entire function $ f $ are defined by
\begin{align*}
 \rho = \limsup_{r \to \infty} \dfrac{\log \log M(r,f)}{\log r},
\end{align*}
\begin{align*}
 \lambda = \liminf_{r \to \infty} \dfrac{\log \log M(r,f)}{\log r}
\end{align*}
and
\[ \tau = \limsup_{r \to \infty} \dfrac{\log M(r,f)}{r^{\rho}}.\]
If $ \tau = 0, $ then $ f $ is said to be of \textit{minimal type}.

It was proved by Baker \cite[Satz 1]{iB58} that a \tef\ $ f $ of order less than $ 1/2 $ satisfies (\ref{maxmin}) for sufficiently large values of~$ C $; this also follows from the version of the $ \cos \pi \rho $ theorem proved by Barry~\cite{Ba63}.  This establishes  Theorem ~\ref{funcs} part (a).

Next we consider functions with suitable gaps in their power series expansions. A \tef\ $ f $ is said to have \textit{Fabry gaps} if
\begin{align}
\label{pwer}
f(z) = \sum_{k = 0}^\infty a_kz^{n_k},
\end{align}
where $ n_k/k \to \infty $ as $ k \to \infty. $ It follows from a result of Fuchs \cite[Theorem 1]{Fu} that, if $ f $ has finite order and Fabry gaps, then for each $ \varepsilon>0 $,
\[ \log m(r) > (1 - \varepsilon) \log M(r), \]
for values of $ r $ outside a set of zero logarithmic density.  It is easy to check that this implies that functions of finite order with Fabry gaps satisfy (\ref{maxmin}) for $ C>1 $, which proves Theorem~\ref{funcs} part~(b).

Hayman \cite[Theorem 3]{Hay} showed that the conclusion of Fuchs' result holds for \tef s of \emph{any} order provided that a stronger gap condition is satisfied, which we call \emph{Hayman gaps}.  The condition is that, in the expansion (\ref{pwer}), we have
\[ n_k > k \log k (\log \log k)^\alpha,\]
for some $ \alpha>2 $ and sufficiently large values of $ k. $  As before, it follows that such functions satisfy (\ref{maxmin}) for $ C>1 $, which proves Theorem~\ref{funcs} part~(c).

Next, we consider functions with the pits effect.  Loosely speaking, a function exhibits the \textit{pits effect} if it has very large modulus except in small regions (pits) around its zeros.  Littlewood and Offord \cite{LO} showed that, if $ \sum_{n = 0}^\infty a_n z^n $ is a \tef\ of order $ \rho \in (0, \infty) $ and lower order $ \lambda > 0 $, and if
\[ S = \left \lbrace f : f(z) = \sum_{n = 0}^\infty \varepsilon_n a_n z^n \right \rbrace \]
where the $ \varepsilon_n $ take the values $ \pm 1 $ with equal probability then, in some precise sense, almost all functions in the set $ S $ exhibit the pits effect. For such functions, it is shown in \cite[Proof of Example 2]{RS10a} that, if $ |z| = r $, then
\[ \log |f(z)| > \frac{1}{4} \log M(r), \]
outside a set of values of~$r$ of finite logarithmic measure.  Again, this is sufficient to show that (\ref{maxmin}) holds for large values of $ C $, and this proves Theorem~\ref{funcs} part~(d).

Finally, we consider functions with multiply connected Fatou components,  that is, multiply connected components of the Fatou set.  The papers \cite{BRS10} and \cite{RS15} give a very detailed analysis of dynamical behaviour in such components, including the following result about the existence of large annuli whose union is forward invariant \cite[Lemma~3.3]{RS15}. Here we use the notation $\delta(r):=1/\sqrt{\log r}$ and also
\[
A(r,R):=\{z:r<|z|<R\},\quad 0<r<R.
\]

\begin{lemma}\label{nested}
Let $f$ be a {\tef} with a {\mconn} Fatou component. Then there exist sequences $(r_n)$ and $(k_n)$, with $r_n > 1$ and $k_n>1$, for $n\in \N_0$, such that the annuli
\[
A'_n=A(r_n^{1+6\pi\delta_n},r_n^{k_n(1-6\pi\delta_n)}),\;\;\text{where }\delta_n=\delta(r_n),\;n\in \N_0,
\]
have the properties that, for $n\in \N_0$,
\[
f(A'_n)\subset A'_{n+1},
\]
and
\[
r_{n+1}= M(r_n) > r_n^{16}.
\]
\end{lemma}
It follows from this lemma that if $r>0$ and $r\in A'_0$, then $m^n(r)\in A'_n$, so
\begin{equation}\label{MCWDminmod}
m^n(r) > r_n^{1+6\pi\delta_n}\ge r_n=M^n(r_0),\qfor n\in \N_0.
\end{equation}
Hence $m^n(r)\to \infty$ as $n\to\infty$, as required.

\section{The sets $ V(f) $ and $\vfp$ : proof of Theorem~\ref{fast}}
\label{setV}
\setcounter{equation}{0}

In this section we prove a number of basic properties of the sets $ V(f) $ and $\vfp$, starting with Theorem~\ref{fast}. Recall that, for a \tef\ such that (\ref{minmodprop}) holds, we define
\[ V(f) = \bigcup_{\ell \in \N_0} f^{-\ell} (V_R(f)), \,\, \,  \textrm{where} \, \, V_R(f) = \{ z: |f^n(z)| \geq \mt^n(R), \textrm{ for } n \in \N\},\]
and $ R>0 $ is such that $ \mt(r) > r $ for $ r \geq R $.

We first show that this definition is unambiguous.  To do this, we return to the question of the rate at which $ m^n(r) $ tends to infinity for a \tef\ satisfying (\ref{minmodprop}).  We showed in Theorem \ref{realslow} that $ m^n(r) $ can tend to infinity arbitrarily slowly.  By contrast, the \emph{fastest} rate at which $ m^n(r) $ can tend to infinity must be limited by the growth of $ \mt^n(r) $, and we now prove Theorem~\ref{fast} which shows that this fastest rate is always attained.

\begin{proof}[Proof of Theorem \ref{fast}]
If $ f $ is a \tef\ such that (\ref{minmodprop}) holds, then it follows from Theorem \ref{fullequiv} that there exist $R>0$ such that
\begin{equation}\label{mtcond1rep}
\mt^n(R) \to \infty \textrm{ as } n \to \infty,
\end{equation}
and $ r \geq R $ such that
\begin{equation}
\label{mtcondrep}
 m^n(r) \geq \mt^n(R), \textrm{ for } n \in \N.
\end{equation}
We now show that the set
\[ V_1(R) = \bigcup_{\ell \in \N_0} f^{-\ell} \left( \{ z: |f^n(z)| \geq \mt^n(R), \textnormal{ for } n \in \N\} \right)\]
is independent of the choice of $ R $ satisfying (\ref{mtcond1rep}) and, for $ r \geq R $  satisfying (\ref{mtcondrep}), is equal to the set
\[ V_2(r) = \bigcup_{\ell \in \N_0} f^{-\ell} \left( \{ z: |f^n(z)| \geq m^n(r), \textnormal{ for } n \in \N\} \right). \]
Suppose $ R'>R, $ where $ R $ satisfies (\ref{mtcond1rep}).  Then since $ \mt^n(R') \geq \mt^n(R) $ for $ n \in \N, $ it is clear that $ V_1(R') \subset V_1(R). $
On the other hand, by (\ref{mtcond1rep}) there exists $ k \in \N $ such that $ \mt^k(R)\geq R', $ so for $ z \in V_1(R) $ there exists $ \ell \in \N_0 $ such that
\[ |f^{n + \ell}(z)| \geq \mt^n(R) \geq \mt^{n-k}(R'), \qfor n \geq k. \]
Hence $ |f^{n + \ell + k}(z)| \geq \mt^n(R') $ for $ n \in \N_0, $ and therefore $ V_1(R) \subset V_1(R'). $  It follows that $ V_1(R) = V_1(R') $ and thus that $  V_1(R) $ is independent of the choice of $ R. $

The proof that $ V_1(R) = V_2(r) $ is similar.  Suppose $ r $ and $ R $ satisfy (\ref{mtcond1rep}) and (\ref{mtcondrep}). Then (\ref{mtcondrep}) clearly implies that $ V_2(r) \subset V_1(R) $ and, since by (\ref{mtcond1rep}) there exists $ j \in \N $ such that $ \mt^j(R)\geq r, $ we deduce as above that $ V_1(R) \subset V_2(r). $
\end{proof}

It follows from Theorem \ref{fast} that, if (\ref{minmodprop}) holds, then the set $ V(f) $
is well defined and independent of $ R $, provided $ R>0 $ is such that (\ref{mtcond1rep}) holds, or equivalently that $ \mt(r) > r $ for $ r \geq R $.

Now recall that
\[ \vfp = \bigcup_{\ell \in \N_0} f^{-\ell} (\vfpr), \]
where
\[
\vfpr = \{ z: \exists \,\, (n_j) \textrm{ such that } |f^{n_j}(z)| \geq \mt^{n_j}(R), \textrm{ for } j \in \N\}.
\]
In this definition, $R>0$ satisfies $ \mt(r) > r $ for $ r \geq R $ and $ (n_j) $ is a strictly increasing sequence of positive integers, which in general depends on~$ z $.   Using an argument similar to the proof of Theorem~\ref{fast}, it is easy to see that  if $ f $ is a \tef\ such that condition (\ref{minmodprop}) holds, and $r \geq R > 0$ satisfy \eqref{mtcond1rep} and \eqref{mtcondrep}, then
\begin{equation}\label{Vpluseqn}
\vfp = \bigcup_{\ell \in \N_0} f^{-\ell}(\{ z: \exists \, (n_j) \textrm{ such that } |f^{n_j}(z)| \geq m^{n_j}(r), \textrm{ for } j \in \N\}).
\end{equation}
Thus, the set $ \vfp $ is also well defined and independent of $ R $, provided $ R>0 $ is such that $ \mt(r) > r $ for $ r \geq R $.

We now show that $V(f)$ and $\vfp$ have some of the basic properties of $I(f)$ proved in \cite{E}.  There, Eremenko showed that
\begin{equation}\label{Iprops}
I(f)\neq \emptyset,\;\; I(f)\cap J(f)\neq\emptyset,\;\;J(f)=\partial I(f),
\end{equation}
and that $\overline{I(f)}$ has no bounded components. (Recall that the fast escaping set $A(f)$ also has the properties listed in~\eqref{Iprops}, and in addition $A(f)$ has no bounded components;  see~\cite{BH99},~\cite{RS05} and~\cite{RS10a}.)  For $V(f)$ and $\vfp$, we have the following.

\begin{theorem}\label{Vprops}

\begin{itemize}
\item[(a)] Let $f$ be a {\tef} such that \eqref{minmodprop} holds.  Then
\begin{equation}\label{Vprops1}
V(f)\neq \emptyset,\;\; V(f)\cap J(f)\neq\emptyset,\;\;J(f)=\overline{V(f)\cap J(f)},\;\;J(f)\subset\partial V(f),
\end{equation}
and $\overline{V(f)}$ has no bounded components.
\item[(b)] If, in addition,
\begin{equation}\label{mtg1}
\liminf_{r\to\infty}\frac{\mt(r)}{r} >1,
\end{equation}
then every Fatou component of~$f$ that meets $V(f)$ must lie entirely in $V(f)$, and $J(f)=\partial V(f)$.
\item[(c)] The properties in~(a) and~(b) hold if $V(f)$ is replaced by $\vfp$.
\end{itemize}
\end{theorem}
We make use of the following result, which is part of \cite[Theorem 3]{RS00}.

\begin{lemma}
\label{dist}
Let~$ f $ be a \tef. If~$ U $ is a simply connected component of~$ F(f) $ and $K$ is a compact subset of $U$, then there exists $ C=C(K) \in (1,\infty) $ such that
\[ \dfrac{|f^n(z_2)|}{|f^n(z_1)|+1} \leq C, \qfor z_1, z_2 \in K \textrm{ and } n \in \N. \]
\end{lemma}

 \begin{proof}[Proof of Theorem \ref{Vprops}]
(a)\;The first three properties of $V(f)$ in \eqref{Vprops1} are immediate since these properties hold for $A(f)$  (see~\cite{RS10a})  and $A(f)\subset V(f)$.

Because $V(f)$ is infinite and completely invariant under~$f$, we have $J(f)\subset \overline{V(f)}$, and this implies that $J(f)\subset \partial V(f)$ since any open subset of $V(f)$ is contained in $F(f)$.

Finally, if $\overline{V(f)}$ has a bounded component, $E$ say, then there is an open topological annulus $A$ which surrounds $E$ and lies in the complement of $\overline{V(f)}$. Since $\overline{V(f)}$ is completely invariant  under $f$ we deduce that $A\subset F(f)$, and since $J(f)\subset\partial V(f)$ it follows  that $A$ is contained in a {\mconn} Fatou component. But any {\mconn} Fatou component of $f$ is contained in $A(f)$  (see~\cite{RS05})  and hence in $V(f)$, so we obtain a contradiction.

(b)\;The first statement of part~(b) is immediate if the Fatou component is multiply connected, since in that case it lies in $A(f)$; see \cite{RS05}. Otherwise this statement follows from Lemma~\ref{dist} and the hypothesis \eqref{mtg1}. Indeed, if~$U$ is a simply connected Fatou component of~$f$  and $z\in U\cap V(f)$, then there exists $\ell\in\N$ such that
\[
|f^{n+\ell}(z)|\ge \mt^n(R),\quad\text{for }n\in\N,
\]
where $ R>0 $ is such that $ \mt(r) > r $ for $ r \geq R $.  We deduce from Lemma~\ref{dist} that for any $z'\in U$ there exists $C(z') > 1$ such that, for each $n \in \N$,
\[
 |f^n(z')| \geq |f^n(z)| / C(z').
\]
The hypothesis \eqref{mtg1} implies that there exists $k=k(C(z'))\in\N$ such that
\[
\mt^{n+k}(R)=\mt^k(\mt^n(R)) > C(z')\mt^n(R),\quad \text{for } n \in \N.
\]
Thus
\[
|f^{n+\ell}(z')| \ge |f^{n+\ell}(z)|/C(z') \ge \mt^n(R)/C(z') \ge \mt^{n-k}(R),\quad \text{for } n > k,
\]
so $z'\in U\cap V(f)$, as required.

This property of Fatou components implies immediately that  $J(f)\supset \partial V(f)$. Since, by part~(a), we also have $J(f)\subset \partial V(f)$, it follows that $J(f)=\partial V(f)$ whenever \eqref{mtg1} holds.

(c)\;  Similar arguments show that the properties in parts (a) and (b) also hold for $\vfp$; we omit the details.
\end{proof}
\begin{remarks*}\normalfont
1.\;Note that, if $ f $ is a \tef\ such that \eqref{minmodprop} holds, then $ \vfp $ is connected by Theorem~\ref{weaksw}, from which it is immediate that $ \vfp $ and $ \overline{\vfp} $ have no bounded components.

2.\;It is natural to ask if the statement that $J(f)\subset \partial V(f)$ in  Theorem \ref{Vprops} part~(a) can be strengthened to $J(f)=\partial V(f)$ for all \tef s that satisfy \eqref{minmodprop}, and similarly for $V^+(f)$.
\end{remarks*}
Finally in this section, we record various relationships involving the connectedness  properties  of $V(f)$, $\vfp$, $I(f)$ and $\kfc$. Recall that a connected set $ E $ is a \sw\ if there exists a sequence $ (G_n) $ of bounded, \sconn\ domains such that
\[
G_n \subset G_{n+1}, \,\, \partial G_n \subset E, \,\, \textrm {for }  n \in \N, \qand \bigcup_{n \in \N} G_n = \C,
\]
and~$E$ is a weak \sw\ if its complement contains no unbounded closed connected sets.
\begin{theorem}
\label{VandI}
Let $f$ be a \tef\ such that \eqref{minmodprop} holds. We have the following implications:
\begin{itemize}
\item[(a)] if $V(f)$ is connected, then $I(f)$ is connected;
\item[(b)] if $V(f)$ is a (weak) spider's web, then $I(f)$ is a (weak) spider's web;
\item[(c)] if $\vfp$ is a spider's web, then $\kfc$ is a spider's web.
\end{itemize}
\end{theorem}
Note that if~\eqref{minmodprop} holds, then both $V^+(f)$ and $I^+(f)$ are weak spiders' webs; see Theorem~\ref{weaksw}.

To prove part~(a) of Theorem~\ref{VandI} we use the following result \cite[Theorem~1.2]{RS14}.
\begin{lemma}\label{EFlemma}
Let~$f$ be a {\tef} and let~$E$ be a set such that $E\subset I(f)$ and $J(f)\subset\overline{E}$. Either $I(f)$ is connected or it has infinitely many components that meet~$E$; in particular, if $E$ is connected, then $I(f)$ is connected.
\end{lemma}
\begin{proof}[Proof of Theorem~\ref{VandI}]  Part~(a) follows from Lemma~\ref{EFlemma} by taking $E=V(f)$, since $J(f)\subset \overline{V(f)}$ by Theorem~\ref{Vprops}.

To prove parts~(b) and~(c) we note that, by the definitions, if a connected set contains a (weak) spider's web, then it is a (weak) spider's web. The fact that $I(f)$ is connected follows from part~(a), and the proof that $\kfc$ is connected was given in \cite[Theorem 1.2]{ORS}.
\end{proof}

\section{Functions for which $ V(f)=A(f) $}
\label{VequalsA}
\setcounter{equation}{0}

In this section we consider functions satisfying \eqref{minmodprop} for which $V(f)=A(f)$. We show that there are many classes of functions with this property. The main  focus  of the section is the proof of Theorem~\ref{erembak}.

Part~(b) of Theorem \ref{erembak} says that, for a \tef\ $ f $ satisfying~(\ref{minmodprop}), if $ V(f) = A(f) $ then both $ V(f) $ and $ I(f) $ are spiders' webs, and~$ F(f) $ has no unbounded components.  This proves that Eremenko's conjecture holds for such functions, and that Baker's conjecture holds for such functions of order less than $ 1/2 $, minimal type.

Part~(a) of the theorem gives a useful equivalent condition that we use in the proof of part~(b), namely that $ V(f) = A(f) $ if and only if
\begin{equation}
\begin{aligned}
\label{C3}
& \textrm{ there exist }  r \geq R > 0 \textrm{ such that } m^n(r) \geq M^n(R), \textrm{ for } n \in \N, \\ & \quad \quad \textrm{and} \,\,\, M^n(R) \to \infty \textrm{ as } n \to \infty.
\end{aligned}
\end{equation}

We need  three  further results for the proof.  The first is part of \cite[Theorem~1.4]{RS15}.

\begin{lemma}
\label{annit}
Let $ f $ be a \tef.  Then there exists $ R_0>0 $ with the property that, whenever $ (a_n) $ is a positive sequence such that
\[ a_n \geq R_0 \qand a_{n+1} \leq M(a_n), \qfor n \in \N_0, \]
there exists $ z \in \C $ and a sequence $ (n_j) $ with $ n_j \to \infty $ as $ j \to \infty $ such that
\[ |f^n(z)| \geq a_n, \qfor n \in \N_0 , \]
and
\[ |f^{n_j}(z)| \leq M^2(a_{n_j}), \qfor j \in \N. \]
\end{lemma}

 Next,  we need the following result from plane topology, which will also be used several times later in the paper.

\begin{lemma}\cite[page~84]{New}\label{Newman}
If $E_0$ is a continuum in $\hat{\C}$, $E_1$ is a closed subset of $E_0$ and $C$ is a component of $E_0\setminus E_1$, then $\overline{C}$ meets $E_1$.
\end{lemma}

 Finally, we need the following sufficient condition for the set $ A_R(f) $ to be a \sw. Recall that
\[ A_R(f) = \{ z: |f^n(z)| \geq M^n(R), \textrm{ for } n \in \N\}, \]
where $R>0$ is such that $M^n(R) \to \infty$ as $n \to \infty$.

\begin{lemma}\cite[Corollary 8.2]{RS10a}\label{swsuff}
Let $ f $ be a \tef\ and let $R>0$ be such that $M^n(R) \to \infty$ as $n \to \infty$.  Then $ A_R(f) $ is a \sw\ if there exists a sequence $ (s_n)_{n \in \N_0} $ such that
\[ s_n \geq M^n(R) \qand m(s_n) \geq s_{n+1}, \qfor n \in \N_0. \]
\end{lemma}

We are now in a position to prove Theorem \ref{erembak}.    

\begin{proof}[Proof of Theorem \ref{erembak}]
We first prove part~(a).  Suppose that~$ f $ is a \tef\  satisfying condition (\ref{minmodprop}), and that (\ref{C3}) also holds.  If $ z \in V(f) $, then it  follows from Theorem~\ref{fast} and (\ref{C3}) that, for some $ \ell \in \N_0 $,
\[
|f^{n+\ell}(z)| \ge \mt^n(r) \geq m^n(r) \geq M^n(R), \textrm{ for } n \in \N,
\]
and so $ z \in A(f). $  Thus $ V(f) \subset A(f). $   As  it is always the case that $ A(f) \subset V(f), $ we have shown that (\ref{C3}) implies $ V(f) = A(f) $.

To establish the converse, we prove the contrapositive.  First note that, since (\ref{minmodprop}) holds, it follows from Theorem~\ref{fast} that there exist $ r \geq R > 0 $ such that
\begin{align}
\label{C5}
 m^n(r) \geq \mt^n(R) \textrm{ for } n \in \N, \qand \mt^n(R) \to \infty \textrm{ as } n \to \infty.
\end{align}
It then follows from Lemma~\ref{annit} that there exists $ z \in \C $ and a sequence $ (n_j) $ with $ n_j \to \infty $ as $ j \to \infty $ such that
\begin{align}
\label{C6}
 |f^n(z)| \geq m^{n}(r), \qfor n \in \N_0 ,
\end{align}
and
\begin{align}
\label{C7}
 |f^{n_j}(z)| \leq M^2(m^{n_j}(r)), \qfor j \in \N.
\end{align}
Suppose now that (\ref{C3}) does \emph{not} hold.  Then for every $ \ell \in \N_0 $, there exists $ N_{\ell} \in \N $ such that
\[ m^{N_{\ell} + \ell}(r) < M^{N_{\ell}}(r), \]
and therefore
\begin{align}
\label{C4}
m^{n + \ell}(r) < M^n(r), \qfor n \geq N_\ell.
\end{align}
Now it follows from (\ref{C5}) and (\ref{C6}) that $ z \in V(f). $  However, by (\ref{C4}), we deduce that, for each $ \ell \in \N $, there exists $ j = j(\ell) $ such that
\[
m^{n_j}(r)<M^{n_j-\ell}(r).
\]
Therefore, by (\ref{C7}),
\begin{align*}
|f^{n_j}(z)| & \leq M^2(m^{n_j}(r))\\
& \leq M^{n_j - \ell+2}(r),
\end{align*}
so $ z \notin A(f). $  Thus $ V(f) \neq A(f) $, and this completes the proof of part (a).


 To prove part~(b), observe that condition (\ref{C3}) implies that the set $ A_R(f) $ is a \sw, by Lemma \ref{swsuff} with $ s_n = m^n(r). $  It now follows from \cite[Theorems 1.4 and 1.5]{RS10a} that $ A(f) $ and $ I(f) $ are also spiders' webs, and that $ f $ has no unbounded Fatou components. Since \eqref{C3} is equivalent to the condition $V(f)=A(f)$ by part~(a), this completes the proof.
\end{proof}

We conclude this section by discussing when the functions included in Theorem~\ref{funcs} have the property that $V(f) = A(f)$. First it is clear, by Theorem~\ref{erembak} part~(a) and \eqref{MCWDminmod}, that functions with a multiply connected Fatou component have this property. The following result shows that the other functions covered by Theorem~\ref{funcs} have the property that $V(f)=A(f)$ provided they also satisfy the weak regularity condition given in part~(b) below. (Recall that we proved in Section~\ref{Functions} that the functions in parts~(a)--(d) of Theorem~\ref{funcs} all satisfy the condition  in Lemma \ref{suffcon}, which is the same as the condition in part~(a) below.)

\begin{theorem}\label{VAfuncs}
Let $f$ be a {\tef} and let $R>0$ be such that $M(r) > r$ for $r\geq R$. Then $V(f) = A(f)$ if, for some $C>1$,
\begin{itemize}
\item[(a)]there exists $R_0 > 0$ such that, for $r \geq R_0$,
\begin{equation}\label{min}
 \mbox{ there exists } s \in (r,r^C) \mbox{ with } m(s) \geq M(r), \mbox{ and }
\end{equation}

\item[(b)] $f$ has regular growth in the sense that there exists a sequence $(r_n)_{n\in \N_0}$ with
    \begin{equation}\label{reg}
     r_n \ge M^n(R) \mbox{ and } M(r_n) \geq r_{n+1}^C, \mbox{ for } n \in \N_0.
    \end{equation}
\end{itemize}
 \end{theorem}
 \begin{proof}
 Let $R_0 > 0$ be as in part~(a). Then, by part~(b), there exists a sequence $(r_n)_{n\in \N_0}$ satisfying~\eqref{reg} with $r_n > R_0$, for $n \in \N_0$. So, by~\eqref{min}, for each $n \in\N_0$, there exists $s_n \in (r_n, r_n^C)$ with
 \[
  m(s_n) \geq M(r_n) \geq r_{n+1}^C > s_{n+1}.
 \]
 Therefore
 \[
 \mt^n(s_0) \geq s_n \geq r_n \geq M^n(R),\qfor  n \in \N_0,
 \]
 so $ \mt^n(s_0) \to \infty $ as $ n \to \infty. $  Thus, by Theorem~\ref{fast}, there exists $r' \geq s_0$ such that $m^n(r') \geq M^n(R)$, for $ n \in \N $. It now  follows from Theorem~\ref{erembak} part~(a) that $V(f) = A(f)$.
 \end{proof}

\section{The iterated minimum modulus and long continua}
\label{WeakSWRes}
\setcounter{equation}{0}

In \cite{ORS} we showed that, for a \tef\ $ f $, the set of points $ \kfc $ at which the iterates of $ f $ are unbounded is connected whenever (\ref{minmodprop}) holds.  The proof of \cite[Theorem~1.2]{ORS} also shows that the complement of $ \kfc $ contains no unbounded closed connected sets, and thus that $ \kfc $ is a weak \sw. This is one part of Theorem \ref{weaksw}.

In this section we prove the other part of Theorem \ref{weaksw}, which states that, if condition (\ref{minmodprop}) holds, then $ \vfp $ is also a weak \sw. In fact, we show that if \eqref{minmodprop} holds, then very many subsets of $\kfc$ are weak \sw s.

For a \tef\ $ f $ such that $m^n(r) \to \infty$ as $n\to\infty$, for some $r>0$, we put
\[ \ifm: = \{z: \textrm{there exists } (n_j) \textrm{ such that } |f^{n_j}(z)| \geq m^{n_j}(r), \textrm{ for } j \in \N \}, \]
where $ (n_j) $ is a strictly increasing sequence of positive integers, which in general depends on~$z$. We prove the following result.

\begin{theorem}
\label{allweaksw}
Let $ f $ be a \tef\ satisfying \eqref{minmodprop} and let $r>0$ be such that $m^n(r) \to \infty$ as $n\to\infty$.  Then each of the sets
\[
\ifm \quad \text{and} \quad \ifmh
\]
is a weak \sw, and hence $\vfp$ is a weak \sw s.
\end{theorem}
The fact that $\vfp$ is a weak \sw\ follows, by property \eqref{Vpluseqn}, from the fact that $\ifmh$ is a weak \sw.

As noted above, the result that $ \kfc $ is a weak \sw\ whenever \eqref{minmodprop} holds was proved in \cite{ORS}.  The proofs given here for the sets $ \ifm $ and $\ifmh$ use a similar approach,  but they are significantly more complicated and they yield more information  about the structure of the set $I^+(f)$. Indeed, recall from Theorem~\ref{realslow} that if \eqref{minmodprop} holds, then there exist values of $r>0$ such that $m^n(r) \to \infty$ more slowly than any given rate. For such~$r$ the set $\ifmh$ is correspondingly larger than $\vfp$, which is therefore in this sense the smallest subset of $\kfc$ having this form.

We deduce Theorem~\ref{allweaksw} from a new fundamental result which states that if $m^n(r) \to \infty$ as $n\to\infty$, where $r>0$, then in a precise sense certain long continua must meet $I^+(f,(m^n(r)))$. Here, and in  what follows, we denote the complement of $ \ifm $ by
\[
\kfm = \{z: |f^{n}(z)| < m^{n}(r) \textrm{ for sufficiently large } n \}.
\]
\begin{theorem}
\label{notlong}
Let $ f $ be a \tef\, satisfying \eqref{minmodprop}, let $r>0$ be such that $m^n(r) \to \infty$ as $n\to\infty$, and put
\begin{equation}\label{Dn}
D_n = \{z \in \C: |z| < m^n(r) \}, \qfor n \in \N_0.
\end{equation}
Suppose that $ \alpha \in \kfm $ and let $ N_0=N_0(\alpha) \in \N $ be such that
\[ f^n(\alpha) \in D_n, \qfor n \geq N_0. \]
If $ K \subset \kfm $ is a continuum, and $ \alpha \in K, $ then
\[ f^n(K) \subset D_n, \qfor n \geq N_0. \]
Moreover, there exists $N_1=N_1(\alpha)$ such that $ K\subset D_{N_1} $.
\end{theorem}

The following lemma contains the induction step used in the proof of Theorem~\ref{notlong}. The idea of the proof of the lemma is similar to that of \cite[Lemma~3.2]{ORS}, though the details are different.

\begin{lemma}
\label{genctm}
Let $ f $ be a \tef\ such that (\ref{minmodprop}) holds and let $D_n$, $n\in \N_0$, be defined as in \eqref{Dn}. Suppose that,  for some $ j \in \N_0 $, there exists $n_j\in \N_0$ and a continuum $ \Gamma_{n_j} $ with the following properties:
\begin{enumerate}[(i)]
\item $ \Gamma_{n_j} \subset f^{n_j}(\kfm) \cap (\C\setminus D_{n_j})$;
\item there is a point $ z_{n_j} \in \Gamma_{n_j} \cap \partial D_{n_j}$;
\item there is a point $ z'_{n_j} \in \Gamma_{n_j} $ such that $ f^n(z'_{n_j}) \in D_{n_j+n} $, for $ n \in \N$.
\end{enumerate}
Then there exists $\, n_{j+1}>n_j $ and a continuum $ \Gamma_{n_{j+1}} \subset f^{n_{j+1} - n_j}(\Gamma_{n_j}) $ such that properties~(i),~(ii) and~(iii) hold with $n_j$ replaced by $n_{j+1}$ throughout.
\end{lemma}

\begin{proof}
Since $z_{n_j}\in \Gamma_{n_j}$, it follows from property~(i) that there exists a minimal integer $ N = N(z_{n_j}) \in \N_0 $ such that
\begin{equation}\label{kfc-cond1}
f^n(z_{n_j})\in D_{n_j+n},\quad\text{for } n > N.
\end{equation}
On the other hand, the properties of the minimum modulus function imply that
\begin{equation}
\label{surr}
f(\partial D_n) \subset \C \setminus D_{n+1}, \text{ for } n \in \N_0,
\end{equation}
so by property~(ii) we have
\begin{equation*}
f(z_{n_j})\in \C\setminus D_{n_j+1}.
\end{equation*}
Hence $N\ge 1$. Now define $n_{j+1}=n_j+N$. Then, by \eqref{kfc-cond1} and the minimality of~$ N $,
\begin{equation}\label{kfc-cond2}
f^n(z_{n_j})\in D_{n_j+n},\quad\text{for } n>n_{j+1}-n_j,
\end{equation}
and
\[
f^{n_{j+1}-n_j}(z_{n_j})\in \C\setminus D_{n_{j+1}}.
\]
Moreover, $f^{n_{j+1}-n_j}(z_{n_j})\notin \partial D_{n_{j+1}}$, by  (\ref{surr})  and \eqref{kfc-cond2}, so
\begin{equation}\label{cond2}
f^{n_{j+1}-n_j}(z_{n_j})\in \C\setminus \overline{D}_{n_{j+1}}.
\end{equation}
Also, by property~(iii),
\begin{equation}\label{cond3}
f^{n_{j+1}-n_j}(z'_{n_j})\in D_{n_{j+1}}.
\end{equation}
It follows from~\eqref{cond2} and~\eqref{cond3} that the continuum $f^{n_{j+1}-n_j}(\Gamma_{n_j})$ includes points from both $D_{n_{j+1}}$ and $\C\setminus \overline{D}_{n_{j+1}}$ (see Figure \ref{LemFig}).


\begin{figure}[htb]
\begin{center}
\includegraphics[width=13cm]{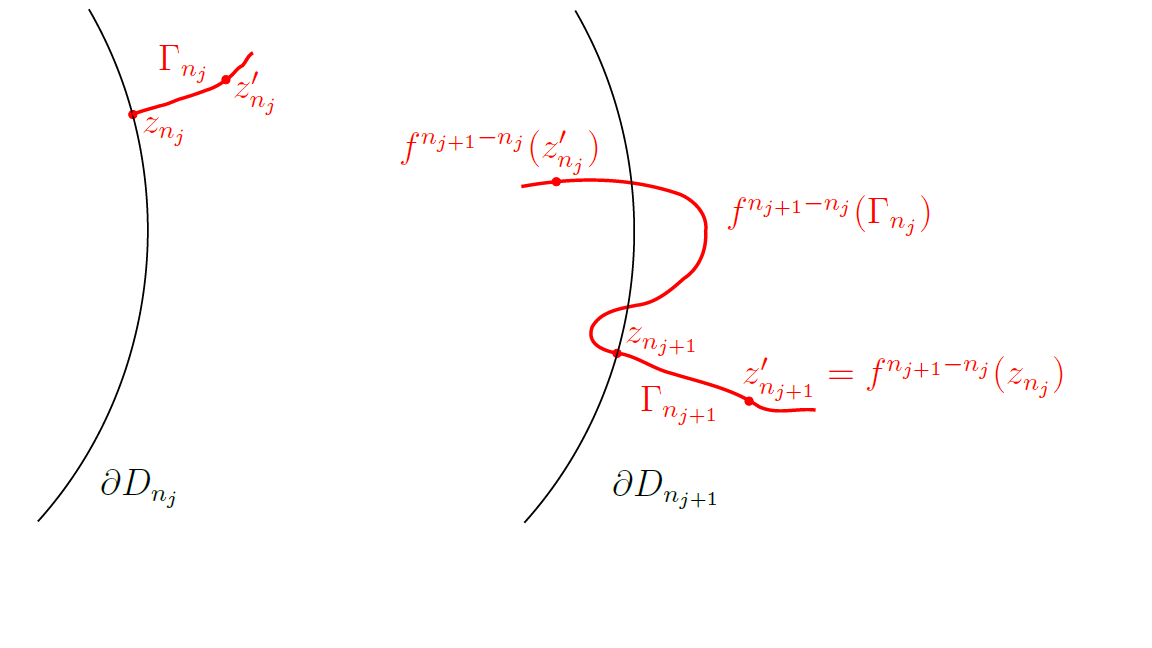}
\vspace{-1.5cm}
\caption{Proof of Lemma~\ref{genctm}}
\label{LemFig}
\end{center}
\end{figure}

Now let $\Gamma_{n_{j+1}}$ be the component of the closed set
\[
f^{n_{j+1}-n_j}(\Gamma_{n_j})\cap (\C\setminus D_{n_{j+1}})
\]
that contains the point
\begin{equation}\label{z'}
z'_{n_{j+1}}:=f^{n_{j+1}-n_j}(z_{n_j}).
\end{equation}
Then
\[ \Gamma_{n_{j+1}} \subset f^{n_{j+1}}(\kfm) \cap (\C\setminus D_{n_{j+1}}), \]
and we deduce that $\Gamma_{n_{j+1}}$ meets $\partial D_{n_{j+1}}$ by applying Lemma~\ref{Newman} with
\[
E_0=f^{n_{j+1}-n_j}(\Gamma_{n_j})\cup \overline{D}_{n_{j+1}} \quad \text{and} \quad E_1 = \overline{D}_{n_{j+1}}.
\]
Thus there exists $z_{n_{j+1}}\in \Gamma_{n_{j+1}}\cap\partial D_{n_{j+1}}$. Therefore, properties~(i) and~(ii) hold with $n_j$ replaced by $n_{j+1}$, and property~(iii) also holds, since
\[
f^n(z'_{n_{j+1}})=f^{n+n_{j+1}-n_j}(z_{n_j})\in D_{n_{j+1}+n},\quad\text{for } n\in\N,
\]
by \eqref{kfc-cond2} and \eqref{z'}.
\end{proof}

Next, we use Lemma \ref{genctm} to prove Theorem~\ref{notlong}.

\begin{proof}[Proof of Theorem~\ref{notlong}]
Let $ \alpha \in \kfm $ and let $N_0 \in \N$ be such that $f^n(\alpha) \in D_n$ for $n \geq N_0$. Also, let $K\subset \kfm$ be a  continuum such that $\alpha \in K$.

We assume for a contradiction that the conclusion of the theorem does not hold; that is, there exists $ N \geq N_0 $ such that $ f^N(K) \cap (\C \setminus D_N) \neq \emptyset. $ We show that this assumption implies that we can select a certain continuum~$ \Gamma $ to act as the starting point for the construction of a sequence $ (\Gamma_{n_j}) $ of continua with the properties stated in Lemma~\ref{genctm}.  This then enables us to obtain the required contradiction.

{\it Step 1: Selection of the continuum $\Gamma$}.  Let $\alpha':=f^N(\alpha)$ and $\Gamma:=f^N(K)$. Then $\Gamma$ is a continuum containing $\alpha'$, and $\alpha'\in D_N$, by hypothesis. Also, $ \Gamma \cap (\C \setminus D_N) \neq \emptyset $, by assumption. Hence, we have
\[
\Gamma \subset f^N(\kfm),\;\; \Gamma \cap \partial D_N\ne \emptyset,\;\; \alpha'\in \Gamma \;\;\text{and } f^n(\alpha')\in D_{N+n},\qfor n\in \N_0.
\]

{\it Step 2: Construction of a sequence of continua $\Gamma_{n_j}$}. We now relabel  $ D_n $ as $ D_{n - N} $ for $ n \geq N $, and put $ r' = m^N(r) $, giving
\[
\Gamma \subset \kfmd,\;\; \Gamma \cap \partial D_0\ne \emptyset,\;\; \alpha'\in \Gamma\;\;\text{and } f^n(\alpha')\in D_{n},\qfor n\in \N_0.
\]
Let $ z_0 \in \Gamma \cap \partial D_0 $.  Then, since $\Gamma \subset \kfmd$, there exists $ n_0 \in \N $ such that
\begin{equation}
\label{alp}
f^{n_0}(z_0) \in \C\setminus D_{n_0} \quad\text{and}\quad f^n(z_0) \in D_n,  \, \text{for } n > n_0.
\end{equation}
Note that $ f^{n_0}(\Gamma) $ meets $ \C\setminus \overline{D}_{n_0} $, since  if $ z_{n_0}' : = f^{n_0}(z_0) $ lay in $ \partial D_{n_0} $ we would have $ f^{n_0+1}(z_0) \in \C \setminus D_{n_0 + 1} $, contradicting (\ref{alp}). Since $ f^{n_0}(\alpha') \in D_{n_0} $,  it follows that $ f^{n_0}(\Gamma) $ meets $ \partial D_{n_0} $.

Now let $ \Gamma_{n_0} $ be the component of $ f^{n_0}(\Gamma) \setminus D_{n_0} $ that contains $ z_{n_0}' $.  Then
\begin{equation}\label{Gamman0}
\Gamma_{n_0}\subset f^{n_0}(\Gamma),
\end{equation}
and $ \Gamma_{n_0} $ meets $ \partial D_{n_0} $, by Lemma \ref{Newman}, applied with
\[
E_0=f^{n_{0}}(\Gamma)\cup \overline{D}_{n_{0}} \quad \text{and} \quad E_1 = \overline{D}_{n_{0}}.
\]
Hence $ \Gamma_{n_0} $ satisfies
\begin{enumerate}[(i)]
\item $ \Gamma_{n_0} \subset f^{n_0}(\kfmd) \cap (\C\setminus D_{n_0})$;
\item there is a point $ z_{n_0} \in \Gamma_{n_0} \cap \partial D_{n_0}$;
\item there is a point $ z'_{n_0} \in \Gamma_{n_0} $ such that $ f^n(z'_{n_0}) \in D_{n_0+n} $, for $ n \in \N$.
\end{enumerate}
Thus by Lemma~\ref{genctm} with $ r $ replaced by $ r', $ there exist a strictly increasing sequence $ (n_j)_{j \in \N_0} $ and a sequence of continua $(\Gamma_{n_j})_{j\in \N_0}$ such that, for each $j\in \N_0$,
\begin{enumerate}[(i)]
\item $ \Gamma_{n_j} \subset f^{n_j}(\kfmd) \cap (\C\setminus D_{n_j})$;
\item there is a point $ z_{n_j} \in \Gamma_{n_j} \cap \partial D_{n_j}; $
\item there is a point $ z'_{n_j} \in \Gamma_{n_j} $ such that  $ f^n(z'_{n_j}) \in D_{n_j+n} $, for $ n \in \N; $
\item $f^{n_{j+1}-n_j}(\Gamma_{n_j})\supset \Gamma_{n_{j+1}}$.
\end{enumerate}

{\it Step 3: Construction of a point in $\Gamma \cap \ifmd$ }.
We now apply Lemma~\ref{RSlemm} with
\[
E_j=\Gamma_{n_j}\quad \text{and} \quad m_j=n_{j+1}-n_j,\qfor j\in\N_0.
\]
By property~(iv) above,
\[
f^{m_j}(E_j)\supset E_{j+1},\qfor j\in \N_0,
\]
and we deduce from Lemma~\ref{RSlemm} that there exists $\zeta\in E_0=\Gamma_{n_0}$ such that
\[
f^{m_0+\cdots +m_k}(\zeta)\in E_{k+1}, \qfor k\in \N_0;
\]
that is,
\[
f^{n_{k+1}-n_0}(\zeta)\in \Gamma_{n_{k+1}}, \qfor k\in \N_0.
\]
Thus, by property~(i) of the sequence of continua $(\Gamma_{n_j})$,
\[
f^{n_{k+1}-n_0}(\zeta) \in \C \setminus D_{n_{k+1}}, \qfor k \in \N_0.
\]
It follows from \eqref{Gamman0} that there exists $ \zeta' \in \Gamma $ such that $ f^{n_0}(\zeta') = \zeta $, and hence
\[
f^{n_{k+1}}(\zeta') \in \C \setminus D_{n_{k+1}}, \qfor k \in \N_0.
\]
Therefore $\zeta' \in \ifmd$, which contradicts the fact that $\zeta'\in \Gamma \subset \kfmd$.

This completes the proof that $ f^n(K) \subset D_n, $ for $ n \geq N_0. $

Finally, if we choose $N_1=N_1(\alpha)$ such that
\[
\{f^k(\alpha):0\le k <N_0\}\subset D_{N_1}\;\;\text{and}\;\;D_{N_0}\subset D_{N_1}\subset D_{N_1+1},
\]
then $K\subset D_{N_1}$, for otherwise we can deduce, by repeatedly applying the minimum modulus property \eqref{minmodprop} to the continua $f^k(K)$, $k=0, \ldots, N_0-1$, that $f^{N_0}(K)$ meets $\partial D_{N_1}$, and hence meets $\partial D_{N_0}$.
\end{proof}

We are now in a position to complete the proof of Theorem \ref{allweaksw}.  To do this, we make use of  the following characterisation of a disconnected subset of the plane.

\begin{lemma}\cite[Lemma 3.1]{R11}
\label{rempe}
A subset $ S $ of $ \C $ is disconnected if and only if there exists a closed, connected set $ \Omega \subset \C $ such that $ S \cap \Omega = \emptyset $ and at least two different components of $ \Omega^c $ intersect $ S. $
\end{lemma}

\begin{proof}[Proof of Theorem~\ref{allweaksw}]
It follows from Theorem~\ref{notlong} and another application of Lemma~\ref{Newman} that $ \kfm $ contains no unbounded closed connected set.  Therefore, to complete the proof that $ \ifm $ is a weak \sw, we must show that this set is connected.

Suppose that $ \ifm $ is disconnected.  Then, by Lemma~\ref{rempe}, there is a closed connected set $ E \subset \kfm $ such that at least two different components of $ E^c $ intersect $ \ifm $. Now $ E $ is bounded by Theorem~\ref{notlong}, so at least one such component of $E^c$, say $ G$, is bounded. Clearly, $G$ is simply connected, so $ \partial G $ is a continuum.

We now show that, as in the proof of Theorem~\ref{notlong}, we can select a certain continuum~$ \Gamma $ to act as the starting point for the construction of a sequence $ (\Gamma_{n_j}) $ of continua with the properties stated in Lemma \ref{genctm}, which leads to a contradiction.  The argument is identical to the proof of Theorem~\ref{notlong} except for the selection of this initial continuum $ \Gamma, $ which we now describe.

By the choice of $ G $ we have
\[ \partial G \subset \kfm \quad \textrm{and} \quad G \cap \ifm \neq \emptyset, \]
so there exist $ \alpha \in \partial G, \, \beta \in G $ and $ N = N(\alpha,\beta) \in \N $ such that
\[ f^n(\alpha) \in D_n, \qfor n \geq N,\]
and
\[ f^N(\beta) \in \C \setminus D_N. \]
Since $ f^N(\alpha) \in f^N(\partial G) $, whereas $ f^N(\beta) $ lies in a bounded complementary component of $ f^N(\partial G) $, it follows that $ f^N(\partial G) $ meets $ \partial D_N. $  We deduce, by applying Lemma~\ref{Newman} and using the fact that $\partial G \subset \kfm$, that the component $ \Gamma $ of $ f^N(\partial G) \cap \overline{D}_N $ that contains $\alpha':=f^N(\alpha) $ is a continuum satisfying
\[
\Gamma \subset f^N(\kfm),\;\; \Gamma \cap \partial D_N\ne \emptyset,\;\; \alpha'\in \Gamma \;\;\text{and } f^n(\alpha')\in D_{N+n},\qfor n\in \N_0.
\]
The proof now proceeds as from Step 2 in the proof of Theorem~\ref{notlong}, and leads to the conclusion that $ \Gamma \cap \ifmd \neq \emptyset $, which is a contradiction.  Thus we have shown that $ \ifm $ is connected, and hence that it is a weak \sw.

The proof that $\ifmh $ is also a weak \sw\ now follows fairly easily. First, this set contains $\ifm$, so there is no unbounded closed connected set in its complement. Thus we just need to show that $\ifmh$ is connected. However, for each $\ell\in\N_0$, the set $f^{-\ell}(\ifm)$ can be shown to be connected by minor modifications  of the above arguments  that $\ifm$ is connected, and so the connectedness of the nested union $\ifmh$ follows.

This completes the proof of Theorem~\ref{allweaksw}.
\end{proof}

In \cite[Theorem 3.1]{ORS}, we showed that $ \kfc $ is a weak \sw\ for functions satisfying more general conditions than (\ref{minmodprop}).  We remark that Theorem \ref{allweaksw} can be generalised in a similar way.  Suppose that $ f $ is a \tef\ and that there exists a sequence of bounded, \sconn\ domains  $ (D_n)_{n \in \N_0} $ such that
\begin{equation}
\label{conditiona}
f(\partial D_n) \text{ surrounds } D_{n+1}, \text{ for } n \in \N_0,
\end{equation}
and
\begin{equation}
\label{conditionb}
\text{every disc centred at } 0 \text{ is contained in } D_n \text{ for sufficiently large  } n.
\end{equation}
These domains $D_n$ generalise the discs given by \eqref{Dn}, which were used in the proof of Theorem~\ref{allweaksw}. Using these domains, we define the set
\[ I^{\!+\!}(f, (D_n)) = \{z: \textrm{there exists } (n_j) \textrm{ such that } f^{n_j}(z) \in \C \setminus D_{n_j}, \textrm{ for } j \in \N \}, \]
where $ (n_j) $ is a strictly increasing sequence of positive integers that, in general, depends on~$ z $.  Then the following result can be proved by making only slight changes to the proof of Theorem \ref{allweaksw}; we omit the details.

\begin{theorem}
\label{fullresult}
Suppose the \tef\ $ f $ and the sequence $ (D_n)_{n \in \N_0} $ of bounded, \sconn\ domains satisfy (\ref{conditiona}) and (\ref{conditionb}).  Then each of the sets
\[
I^{\!+\!}(f, (D_n)) \quad \text{and} \quad  \bigcup_{\ell \in \N_0} f^{-\ell}(I^{\!+\!}(f, (D_n)))
\]
is a weak \sw.
\end{theorem}

We conclude this section by using Theorem \ref{fullequiv} to show that if \eqref{minmodprop} holds and there exists an invariant curve under~$f$, which tends to $\infty$, then this curve must contain a point of $I(f)$.

\begin{theorem}\label{inv}
Let $ f $ be a \tef\ such that condition (\ref{minmodprop}) holds and let~$ \Gamma $ be a simple curve tending to $\infty$ and invariant under $ f $. Then $ \Gamma \cap I(f) \neq \emptyset. $
\end{theorem}

\begin{proof}
Let $ \psi: \Gamma \to [0, \infty) $ be a homeomorphism and define
\[
\varphi(t) = \psi \circ f \circ \psi^{-1} (t) \qand \widetilde{\varphi}(t) := \max_{0 \leq s \leq t} \varphi(s), \qfor t \in [0, \infty).
\]
We claim that there exists $t_0$ such that $\varphi^n(t_0) \to \infty$ as $n \to \infty$.  The theorem follows from the claim since, if $ z_0 = \psi^{-1}(t_0) $, then
\[f^n(z_0) = f^n(\psi^{-1}(t_0)) = \psi^{-1}(\varphi^n(t_0)) \to \infty  \qas n \to \infty, \]
so $ z_0 \in  \Gamma \cap I(f). $

To prove the claim we use Theorem~\ref{fullequiv} to choose $ r \geq R > 0 $ such that
\[ m^n(r) \textrm{ and } \mt^n(R) \textrm{ increase strictly with } n \textrm{ to } \infty, \]
and
\[ m^n(r) \in [ \, \mt^n(R) , \mt^{n+1}(R) \, ], \qfor n \in \N_0. \]
Without loss of generality, we can assume that $ \Gamma $ meets each of the circles $ \{z: |z| = m^n(r) \} $, for $ n \in \N_0. $ Then there exists $T \geq 0 $ such that, for each $ t \geq T, $
\begin{align*}
m^n(r) \leq | \psi^{-1}(t)| < m^{n+1}(r), \qfor \textrm{some } n \in \N_0.
\end{align*}
Moreover, for $t\ge T$, there exists $ z_t \in \Gamma $ and a maximal $ N \in \N_0 $ such that
\begin{align}
\label{maxn}
|z_t| = m^N(r) \qand \psi(z_t) \leq t.
\end{align}
Then
\[ |f(z_t)| \geq m^{N+1}(r) \]
and also, because $\Gamma$ is invariant under~$f$ and~$N$ is maximal for \eqref{maxn},
\[ \widetilde{\varphi}(t) \geq \varphi(\psi(z_t)) = \psi(f(z_t)) > t. \]
Thus we have shown that $ \widetilde{\varphi}(t) > t $ for $ t \geq T $. Hence, by Theorem~\ref{fullequiv} there exists $ t_0>0 $ such that $ \varphi^n(t_0) \to \infty $ as $ n \to \infty $.  This completes the proof.
\end{proof}
\begin{remark*}
It follows from the result of Short and Sixsmith \cite[Theorem 1.6]{SS}, mentioned in Section~2, that the conclusion of Theorem~\ref{inv} can be strengthened to state that $\Gamma$ contains uncountably many points in $I(f)$.
\end{remark*}

\section{Unbounded Fatou components}
\label{unbFat}
\setcounter{equation}{0}
In this section we prove Theorems~\ref{fatou} and~\ref{BD}  which, for a \tef\ $f$ for which the condition (\ref{minmodprop}) holds, concern the relationship between {\it unbounded} Fatou components and the sets $\vfp$ and $V(f)$.  We begin by noting that if a point $z$ belongs to a Fatou component of $f$ that lies in an attracting or parabolic basin, or a Siegel disc (or a preimage of a Siegel disc), then the orbit of~$z$ must be bounded. Thus any Fatou component of $f$ that meets $\vfp$ must be a Baker domain (or a preimage of a Baker domain) or a wandering domain.

Recall that, if $ U = U_0 $ is a Fatou component, then $ f^n(U) \subset U_n $ for some Fatou component $ U_n $, for each $ n \in \N $.  A Fatou component of a \tef\ is called a \textit{\wand} if it is not eventually periodic, that is, if $ U_m \neq U_n $ for $ m \neq n $, and it is called a {\it Baker domain} if it is periodic and $f^n(z) \to \infty$ as $n \to \infty$ for $z \in U$. Note that Baker domains of {\tef}s are always unbounded whereas wandering domains can be either bounded or unbounded. (For a full description of the possible types of Fatou components of a \tef\ see, for example,~\cite{wB93}.)

\begin{proof}[Proof of Theorem \ref{fatou}]
(a) Let $f$ be a \tef\ satisfying the condition \eqref{minmodprop} and let~$U$ be an unbounded Fatou component of~$f$. By  Theorem~\ref{fast}  we can take $r\ge R>0$ such that
\[
m^n(r)\ge \mt^n(R),\qfor n\in\N_0,\quad\text{and}\quad \mt^n(R)\to\infty \;\;\text{as } n\to \infty.
\]
Then we can use Theorem~\ref{notlong} to show that $U \cap \vfp \neq \emptyset$. Indeed, if this is false, then  $U\subset K(f,(m^n(r)))$ and, by Theorem~\ref{notlong}, for any $\alpha\in U$ there exists $N_1=N_1(\alpha)$ such that any continuum $K$ in $U$ that contains~$\alpha$ must lie in $D_{N_1}$, which is a contradiction.  Therefore, we must have $U \cap \vfp \neq \emptyset$. As noted above, this implies that~$U$ is either a Baker domain (or a preimage of a Baker domain) or a wandering domain.

The second statement of part~(a)  follows from Theorem~\ref{Vprops} part~(c).

(b) Zheng's result~(b) before the statement of Theorem~\ref{fatou} states that if
\[
\limsup_{r \to \infty} \frac{m(r)}{r} = \infty,
\]
then any unbounded component $ U $ of $ F(f) $ must be a \wand. This condition is clearly satisfied if $\lim_{r \to \infty} \mt(r) / r = \infty$. The fact that such a wandering domain must be in $\vfp$ follows from part~(a).

(c) The fact that $U$ is a wandering domain in $\vfp$ follows from part~(b), and $\vfp\subset Z^+(f)$ because the hypothesis about $\mt(r)$ implies that,  for sufficiently large $ r, $
\[
\frac{\log^+ \log^+ \mt^n(r)}{n}\to \infty\;\;\text{as}\;\; n\to \infty,
\]
by Lemma \ref{suffcon}.
\end{proof}

Finally in this section, we prove our result about the rate of escape in certain Baker domains.

\begin{proof}[Proof of Theorem \ref{BD}]
(a) Suppose that \eqref{minmodprop} holds and $U$ is an invariant Baker domain of the \tef\ $f$. By Theorem~\ref{fullequiv}, we can take $ r \ge R > 0 $ such that
\begin{equation}\label{mandmt}
m^n(r)\ge \mt^n(R),\qfor n\in\N_0,\quad\text{and}\quad \mt^n(R)\to\infty \;\;\text{as } n\to \infty.
\end{equation}
and
\[ m^n(r) \textrm{ increases strictly with } n \;(\textrm{to } \infty). \]
We now take $\alpha\in U$ and let $\Gamma$ be a compact curve in $U$ joining $\alpha$ to $f(\alpha)$. Moreover, we assume without loss of generality (by choosing $\alpha$ suitably and replacing~$r$ by some iterate $m^n(r)$ if necessary) that
\[
 |\alpha| \leq r < |f(\alpha)|.
\]
Now let $n_0$ denote the largest value of $n \in \N_0$ such that
\begin{equation}
 \label{n0}
 \Gamma \cap \{ z: |z| = m^{n}(r) \} \neq \emptyset.
 \end{equation}
 Then
 \begin{equation}
 \label{f1}
f(\Gamma) \cap \{z: |z| \geq m^{n_0+1}(r)\} \neq \emptyset.
 \end{equation}
 Also, $f(\alpha) \in \Gamma \cap f(\Gamma)$, so  by the definition of $n_0$ and the fact that $(m^n(r))$ is strictly increasing, we have
 \begin{equation}
 \label{f2}
f(\Gamma) \cap \{z: |z| < m^{n_0+1}(r)\} \neq \emptyset.
 \end{equation}
It follows from \eqref{f1} and \eqref{f2} that, if we let $n_1$ denote the largest value of $n \in \N_0$ such that
 \begin{equation}
 \label{n1}
f(\Gamma) \cap \{ z: |z| = m^{n}(r) \} \neq \emptyset,
 \end{equation}
 then $n_1 \geq n_0 + 1 \geq 1$.

By repeating this process we find that, for each $k \in \N_0$, there exists $n_k \geq k$ such that
\begin{equation}
 \label{nk}
 f^k(\Gamma) \cap \{ z: |z| = m^{n_k}(r) \} \neq \emptyset.
\end{equation}

It now follows from Lemma~\ref{dist} that there exists $C=C_{\Gamma} > 1$ such that, for each $z \in \Gamma$ and each $k \in \N_0$,
\[
 |f^k(z)| \geq m^k(r) / C.
\]
Finally, by applying Lemma~\ref{dist} again, we deduce from \eqref{mandmt} that for any $z \in U$ there exists $C(z) > 1$ such that, for each $n \in \N$,
\[
 |f^n(z)| \geq m^n(r) / C(z)\ge \mt^n(R)/C(z).
\]
(b) If we also have
\[
\liminf_{r\to\infty} \frac{\mt(r)}{r} >1,
\]
then, for any $z\in U$ there exists $k=k(C(z))\in \N$ such that
\[
\mt^{n+k}(R)\ge C(z)\mt^n(R),\qfor n\in \N_0,
\]
where $R$ was defined in part~(a),
and hence
\[
|f^{n+k}(z)| \ge \mt^n(R),\qfor n\in \N_0;
\]
that is, $z\in V(f)$.
\end{proof}

\section{Examples}
\label{examples}
\setcounter{equation}{0}

In this section we illustrate our results with examples of fairly simple entire functions, some of which satisfy \eqref{minmodprop} and some of which do not.

Our first example is a function $ f $ of order $ 1 $ which does not belong to the classes of functions covered by Theorem \ref{funcs}.   We show that this function satisfies condition (\ref{minmodprop}) and has an invariant Baker domain that lies in~$ V(f) $.

\begin{example}
\label{Exfirst}
Let $ f $ be the function \[ f(z) = 2z(1+e^{-z}). \]  Then
\begin{enumerate}[(a)]
\item there exists $ r>0 $ such that $ m^n(r) \to \infty $ as $ n \to \infty; $
\item $ f $ has an invariant Baker domain that lies in $ V(f). $
\end{enumerate}
\end{example}

\begin{proof}
Put $ z = re^{i\theta}, $ so
\[
f(re^{i\theta}) = 2r e^{i\theta} ( 1 + e^{-r \cos \theta} e^{-ir \sin \theta} ),\]
and therefore
\begin{align*}
\mu(r, \theta) : = | f(re^{i\theta}) |^2 & = 4r^2 \left( (1 + e^{-r \cos \theta}\cos(r \sin \theta))^2 + (e^{-r \cos \theta}\sin(r \sin \theta))^2 \right)\\
& = 4r^2 \left( 1 + 2e^{-r \cos \theta}\cos(r \sin \theta) + e^{-2r \cos \theta} \right)\\
& = 4r^2 (1 - e^{-r \cos \theta})^2 + 8r^2e^{-r \cos \theta}(1 + \cos(r \sin \theta)).
\end{align*}
Note that both terms in the last line are non-negative for all values of~$ r $ and~$ \theta. $

Now let $ r_n = 2n \pi, $ for $ n \geq 2, $ so $ r_n \to \infty $ as $ n \to \infty $, and $ r_{n+1} \leq \dfrac{3r_n}{2}.$  We claim that
\begin{equation*}
\mu(r_n, \theta) \geq \left( \dfrac{3r_n}{2} \right)^2 \geq r_{n+1}^2, \qfor n \geq 2,
\end{equation*}
from which it follows that $ m(r_n) \geq 3r_n/2 \geq r_{n+1} $ for $ n \geq 2 $, and this proves part~(a) by Theorem \ref{fullequiv}.

To prove the claim, first observe that, if $ e^{-r \cos \theta } \leq \frac{1}{4} $ or $ e^{-r \cos \theta } \geq \frac{7}{4} $, then
\begin{align*}
\mu(r, \theta) \geq 4r^2(1 - e^{-r \cos \theta})^2 \geq 4r^2 \left( \dfrac{3}{4} \right)^2 = \left( \dfrac{3r}{2} \right)^2.
\end{align*}
Suppose therefore that $\frac{1}{4} < e^{-r \cos \theta } < \frac{7}{4}. $  Then
\[ - \dfrac{1}{r}\log \frac{7}{4} < \cos \theta < - \dfrac{1}{r}\log \frac{1}{4}, \]
so we can put $\theta = \tfrac{1}{2}\pi + \varepsilon(r)$,  where $ |\varepsilon(r)| \leq C/r $ for some positive absolute constant $ C. $
Thus, for $ n \geq 2 $ we have
\begin{align*}
r_n \sin \theta & = r_n \left( 1 - \dfrac{\textit{O}(1)}{r_n^2} \right ) = 2n\pi - \dfrac{\textit{O}(1)}{2n\pi},
\end{align*}
and hence
\[ \cos (r_n \sin \theta) = 1 - \dfrac{\textit{O}(1)}{(2n\pi)^2}. \]
It follows that, for large $ n $,
\begin{align*}
\mu(r_n, \theta) \geq 8r_n^2e^{-r_n \cos \theta}(1 + \cos(r_n \sin \theta))
\geq 2r_n^2\left( 2 - \dfrac{\textit{O}(1)}{(2n\pi)^2} \right) \geq \left( \dfrac{3r_n}{2} \right)^2,
\end{align*}
and this completes the proof of the claim, and of part (a).

That $ f $ has an invariant Baker domain $ U $ follows from \cite[Theorem 2]{RS99}, which describes a large family of entire functions with Baker domains, including this  function.   Since we have just shown that for the sequence $ r_n=2n\pi $ we have $ m(r_n) \geq 3r_n/2 $, it follows that $ \liminf_{r \to \infty} \mt(r)/r > 1 $ and thus that $ U \subset V(f) $ by Theorem~\ref{BD} part~(b).
\end{proof}

Our next example concerns the \tef\ $ f(z) = z + 1 + e^{-z} $, first investigated by Fatou \cite{F} and often named after him.  For this function, it is known that $ F(f) $ is a completely invariant Baker domain, that $ I(f) $ is a \sw\ but $ A(f) $ is not, and that $ f $ is \spl; see \cite[Theorem~1.1]{VE} and \cite[Example~5.4]{O12a} for these results and an explanation of the terminology.  We use Theorem~\ref{fullequiv} to show that, nevertheless, condition (\ref{minmodprop}) does not hold for this function~$ f. $

\begin{example}
Let $ f $ be the Fatou function, \[ f(z) = z + 1 + e^{-z}. \]  Then there does not exist $ r>0 $ such that $ m^n(r) \to \infty $ as $ n \to \infty.$
\end{example}

\begin{proof}
We will show that $ \mt(r) < r $ for arbitrarily large $ r $, from which the result follows by Theorem \ref{fullequiv}.

Consider the images under $ f $ of points $ ir, \, r>0. $ As $ r $ increases, the image points travel clockwise around a circle of radius $ 1 $, whose centre is at the same time moving up the line $ \textrm{Re } z = 1.$   Clearly
\[ |f(ir)| = r, \qfor r = (2k+1)\pi, \quad k \in \N, \]
and it is easy to see that there exists $ \varepsilon_k > 0 $, with $ \varepsilon_k \to 0 $ as $ r \to \infty, $ such that
\begin{align}
\label{fat1}
m(r) \leq |f(ir)| \leq r, \qfor 2k\pi + \varepsilon_k \leq r \leq (2k+1)\pi, \quad k \in \N.
\end{align}
Moreover, for $ r>0, $ we have
\begin{align}
\label{fat2}
|f(ir)| \leq r + 1 + |e^{-ir}| \leq r + 2.
\end{align}
Now let $ r_k = (2k+1)\pi - \delta_k $, where $ \delta_k>0 $.  Then for $ \delta_k $ sufficiently small we have
\begin{align*}
m(s) \leq r_k, \qfor 2k\pi + \varepsilon_k \leq s \leq r_k,
\end{align*}
by (\ref{fat1}), and
\begin{align*}
\max {\{m(s): 0 \leq s \leq 2k\pi + \varepsilon_k \}} \, \leq \, 2k\pi + \varepsilon_k + 2 \, < \, r_k,
\end{align*}
by (\ref{fat2}). Thus we have shown that
\[ \mt(r_k) = \max {\{m(s): 0 \leq s \leq r_k \}} < r_k, \qfor  k \in \N, \]
and this completes the proof.
\end{proof}
We remark that a similar argument shows that \eqref{minmodprop} fails for every function of the form $f(z)=z+a+be^{-z},$ where $a,b>0$.

The following example shows that the condition \eqref{minmodprop} can hold even if we have
\begin{equation}\label{limit1}
\lim_{r\to\infty}\frac{\mt(r)}{r}=1.
\end{equation}

\begin{example}
\label{z+bsinz}
Condition \eqref{minmodprop} holds for any function of the form
\[
f(z)=z+b\sin z,\quad\text{where } b>2\pi.
\]
\end{example}

\begin{proof} We claim that such a function~$f$ has the property that if $r_n=2n\pi+\pi/2$, where $n\in \N$, then
\[
m(r_n)\ge r_n+2\pi = r_{n+1}, \quad \text{for sufficiently large } n,
\]
from which it follows that condition (\ref{minmodprop}) holds, by Theorem~\ref{fullequiv}.  It is also clear that \eqref{limit1} holds for functions of this form.

To prove the claim, suppose first that $z=x+iy$, where $x,y \ge 0$, $|z|=r_n$, for some $n\in\N$, and $y\ge \log (3r_n)$. Then  clearly
\[
\sin^2 x+\sinh^2y \ge \sinh^2y \ge r_n^2,
\]
and so
\[
|\sin z|\ge |z|.
\]
Hence,  for such $z$ we have
\begin{equation}\label{est1}
|f(z)|\ge b|\sin z|-|z| \ge 3|\sin z|-|z| \ge 2|z| \ge |z|+2\pi.
\end{equation}

Next suppose that $z=x+iy$, where $x,y \ge 0$ and $|z|=r_n$, for some  $n\in\N$, and $y\le \log (3r_n)$. Then
\[
x^2=r_n^2-y^2 \ge r_n^2-(\log(3r_n))^2 = r_n^2\left(1-\left(\frac{\log(3r_n)}{r_n}\right)^2\right),
\]
so
\[
r_n\ge x\ge r_n\left(1-\left(\frac{\log(3r_n)}{r_n}\right)^2\right)^{1/2}\ge r_n-\frac{(\log(3r_n))^2}{r_n}.
\]
Thus, for such $z$, we have (since $r_n=2n\pi+\pi/2$)
\[
\textrm{Re\,}(\sin z) = \sin x\cosh y \ge \sin x \ge 1- \frac{(\log(3r_n))^2}{r_n},
\]
from which it follows that
\begin{equation}\label{est2}
|f(z)|=|z+b\sin z|\ge r_n+b-o(1) \quad\text{as } n\to \infty,
\end{equation}
uniformly for such $z$.

The claim about $m(r_n)$ now follows from \eqref{est1} and \eqref{est2}, together with the symmetry properties of the sine function.
\end{proof}

It follows from Theorem \ref{funcs} that condition (\ref{minmodprop}) always holds for \tef s of order less than $ 1/2 $. Our final example shows that this condition may or may not hold for \tef s of order $ 1/2 $.

\begin{example}
\label{orderhalf}
\begin{enumerate}[(a)]
\item Condition (\ref{minmodprop}) does not hold for the function $ f(z) = \cos \sqrt z $.
\item Condition (\ref{minmodprop}) holds for the function $ g(z) = 2z \cos \sqrt z $.
\end{enumerate}
\end{example}

\begin{proof}
(a) Since $ f $ is bounded on the positive real axis, it follows that $ m(r) $ is bounded and thus that there is no $ r>0 $ such that $ m^n(r) \to \infty $ as $ n \to \infty. $

(b) First note that the minimum modulus of $ g $ is attained on the positive real axis for every $ r>0 $. This can be deduced, for example, from the fact that~$g$ is an entire function of order $1/2$ with its zeros on the positive real axis, and so (see \cite{Ti}) can be written in the form
\[
g(z)=2z\prod_{n=0}^{\infty}\left(1-\frac{z}{(n+1/2)^2\pi^2}\right).
\]
Now suppose that $ r \geq 9\pi^2. $ Then  $ r \in [n^2\pi^2, (n+1)^2\pi^2 ) $  for some integer $ n \geq 3 $, and it follows that
\[ \mt(r) \geq m(n^2\pi^2) = 2n^2\pi^2 > (n+1)^2\pi^2  > r.  \]
Thus we have shown that $ \mt(r) > r $ for $ r \geq 9\pi^2, $ so it follows from Theorem~ \ref{fullequiv} that condition (\ref{minmodprop}) is satisfied.
\end{proof}

\begin{remarks*}\normalfont
1. Example~\ref{orderhalf} can be modified to apply to \tef s of any positive integer order $ p $, giving that condition (\ref{minmodprop}) is not satisfied for the function $ f(z) = \cos z^p $, but is satisfied for the function $ g(z) = 2z \cos z^p $.

2. The case of functions of order $ 1/2 $, \emph{minimal type}, mentioned in the statement of Baker's conjecture, is rather delicate.  In forthcoming work we will show, using a method pioneered by Kjellberg \cite[page~821]{wH89}, that there are examples of such functions that do satisfy condition (\ref{minmodprop}) and other examples that do not.
\end{remarks*}

\end{document}